\documentclass[final,12pt]{colt2026} 

\title[Theoretical  guarantees for change localization using conformal $p$-values]{Theoretical  guarantees for change localization\\ using conformal $p$-values}
\usepackage{times}
\usepackage{cleveref}
\usepackage{comment}
\usepackage{float}
\newtheorem{assumption}{Assumption}

\newtheorem{fact}{Fact}[section]

\def\E{\mathbb{E}}

\def\P{\mathbb{P}}
\def\X{\mathbf{X}}

\def\N{\mathbb{N}}

\def\R{\mathbb{R}}

\newcommand{\convD}{\overset{d}{\to}}
\newcommand{\convP}{\xrightarrow[]{\mathcal{P}}}
\newcommand{\convAS}{\xrightarrow[]{a.s.}}

\renewcommand{\P}{\mathbb{P}}

\renewcommand{\vec}[1]{\mathbf{#1}}

\newcommand{\RNum}[1]{\uppercase\expandafter{\romannumeral #1\relax}}

\DeclareMathOperator*{\argmax}{arg\,max}

\newcommand{\iid}{\overset{\text{i.i.d}}{\sim}}

\newcommand{\qed}{\hfill\hbox{\vrule height1.5ex width.5em}}

\let\oldsup\sup
\renewcommand{\sup}{\oldsup\limits}
\let\oldmax\max
\renewcommand{\max}{\oldmax\limits}

\coltauthor{%
 \Name{Swapnaneel Bhattacharyya} \Email{swapnaneelbhattacharyya@gmail.com}\\
 \addr Indian Statistical Institute, Kolkata
 \AND
 \Name{Aaditya Ramdas} \Email{aramdas@stat.cmu.edu}\\
 \addr Carnegie Mellon University
}

\begin{document}
\maketitle
\begin{abstract}%
  {\footnotesize Changepoint localization aims to provide confidence sets for a changepoint (if one exists). Existing methods either rely on strong parametric assumptions or provide only asymptotic guarantees, or focus on a particular kind of change (e.g., mean shift) rather than the entire distributional change. The first method to achieve distribution-free changepoint localization with finite-sample validity was recently introduced by \cite{dandapanthula2025conformal}. However, while they proved finite sample coverage, there was no analysis of set size. In this work, we provide rigorous theoretical guarantees for their algorithm. We also show the consistency of a point estimator for change, and derive its convergence rate without distributional assumptions. Along that line, we also construct a distribution-free consistent test to assess whether a particular time point is a changepoint or not. Thus, our work provides unified distribution-free guarantees for changepoint detection, localization, and testing. In addition, we present various finite sample and asymptotic properties of the conformal $p$-value in the distribution change setup, which provides a theoretical foundation for many applications of the conformal $p$-value. As an application, we construct distribution-free consistent tests for exchangeability against distribution-change alternatives and a new, computationally tractable method of optimizing the powers of conformal tests. We run detailed simulations to corroborate the performance of our methods and theory. Together, our contributions offer a comprehensive and theoretically principled approach to distribution-free changepoint inference, broadening both the scope and credibility of conformal methods in modern changepoint analysis. }
\end{abstract}


\section{Introduction}

In the offline changepoint detection problem, one is presented with an ordered dataset, together with the possibility that the underlying data-generating distribution changes at some unknown point in time. This unknown index is commonly referred to as the \textit{changepoint}. The goal of this problem is to identify whether a change has occurred, and if so, then we need to identify the changepoint. For example, we observe ordered data $\mathbf{X} = (X_1,\cdots,X_n)$ such that for some distributions $R,Q$ and unknown $\tau_n \in [n-1], \quad X_1,\cdots,X_{\tau_n} \iid R$ and $X_{\tau_n+1},\cdots,X_n \iid Q$ where $R \neq Q$. Then the problem aims to identify $\tau_n$ based on the observed data. Furthermore, the \textit{distribution-free} changepoint detection problem refers to the particular version of the aforementioned problem, when there is no knowledge of $R$ and $Q$. Similarly, the changepoint localization problem refers to constructing a confidence set for the original unknown changepoint $\tau_n$ maintaining a desired coverage guarantee for finite samples. Finally, the changepoint testing problem refers to assessing whether a particular time point is the true changepoint or not, from the sample. For brevity, we collectively refer to all of these three problems as \textit{changepoint analysis} problem,  formally introduced in \Cref{subsec: probolem formulation}.

The practical significance of changepoint localization is paramount in domains with unknown, evolving distributions where parametric models fail due to environmental variability. In such settings, a finite-sample confidence set for the changepoint $\tau$ is crucial for public safety and accountability, defining precise windows for advisories and regulatory enforcement. Similarly, assessing specific suspected time points requires powerful, distribution-free tests. Existing methods, however, predominantly rely on restrictive parametric assumptions, asymptotic guarantees, or particular change types (e.g., mean shift). While \cite{dandapanthula2025conformal} introduced the first assumption-free algorithm with finite-sample validity, they provided no theoretical analysis of the confidence set size. This leaves open the possibility of trivial set sizes (e.g., $\{1, \dots, n-1\}$), limiting the algorithm's established practical utility.

This work establishes the first theoretical guarantees for the aforementioned localization algorithm, analyzing its set size to validate it as a practically implementable, distribution-free method. We concurrently propose a novel, assumption-free point estimator for $\tau$, deriving its consistency and convergence rates, and construct distribution-free tests for assessing candidate changepoints. Collectively, these results provide a unified, computationally tractable framework for changepoint detection, localization, and testing, backed by rigorous finite-sample and asymptotic theory.

\paragraph{Notation.}
For a random variable $Y$, $F_Y$ denotes its cumulative distribution function (CDF) and $F_Y^{-1}(y) := \inf \{x: F_Y(x) \geqslant y\}$ denotes its inverse. For a natural number $n$, $[n]$ denotes the set $\{1,2,\cdots,n\}$. We denote the collection of all Borel sets on $\R$ by $\mathcal{B}(\R)$ and define $\mathcal{B}[a,b] := \{A : A \subseteq [a,b], A \in \mathcal{B}(\R)\}$. For data $X_1,\cdots,X_n$ we denote their order statistics by $(X_{(1)},\cdots,X_{(n)})$ satisfying $X_{(1)} \leqslant \cdots \leqslant X_{(n)}$. For two CDFs $R,Q,$ 
$R \ll Q$ denotes that $R$ is absolutely continuous with respect to $Q$. For a random variable $X$ having distribution $Q$, we denote its expectation by $\E_Q[Z]$. For random vectors $\mathbf{Y} := (Y_1,\cdots,Y_m)$ and $\mathbf{Z} := (Z_1,\cdots,Z_n)$, we write $\mathbf{Y} \perp \mathbf{Z}$ if  $\mathbf{Y}$ and $\mathbf{Z}$ are independent. Furthermore, $X \overset{d}{=}Y$ means $X,Y$ have identical distribution.

\subsection{Distribution-free change detection, localization \& testing using conformal $p$-values}
\label{subsec: probolem formulation}
We formalize the problem for an ordered sequence $\mathbf{X}_n := (X_1,\cdots,X_n)$ taking values in $\mathcal{X}$. For an unknown $\tau_n \in [n-1]$, we assume the pre-change samples $(X_1,\cdots,X_{\tau_n}) \sim \P_0$ and post-change samples $(X_{\tau_n+1},\cdots,X_n) \sim \P_1$ satisfy the following:

\begin{assumption}[Independence and exchangeabilty of Pre and Post-change data]
\label{indep pre and post}
    $(X_1,\cdots,X_{\tau_n})$ and $(X_{\tau_n+1},\cdots,X_n)$ are independent of each other and are exchangeable (within themselves).
\end{assumption}
Recall that random variables $(X_1,\cdots,X_n)$ are exchangeable, if for any permutation $\pi$ of $[n]$, $(X_1,\cdots,X_n) \overset{d}{=} (X_{\pi(1)},\cdots,X_{\pi(n)})$. Let $\mathcal{H}_{0,t}$ be the hypothesis that $\tau_n = t$ and \Cref{indep pre and post} holds. A set $C_{1-\alpha} \equiv C_{1-\alpha}(\mathbf{X}_n)$ is a \emph{distribution-free confidence set} for $\tau_n$ at level $(1-\alpha)$ if, for all distributions satisfying \Cref{indep pre and post}, $\P_{\mathcal{H}_{0,\tau_n}}\big(\tau_n \in C_{1-\alpha}(\mathbf{X}_n)\big) \geq 1 - \alpha.$ Changepoint detection refers to obtaining a point-estimator $\widehat{\tau}_n \equiv \widehat{\tau}_n(\mathbf{X}_n)$. To discuss asymptotic properties, we consider a triangular array $\{X_{n,j}\}_{n \in \N, 1 \leq j \leq n}$ where the $n$th row $\mathbf{X}_n$ is defined over a common sample space $\Omega_n$. We assume for all $n$, there is $\tau_n \in [n-1]$ such that the pre-change segment $(X_{n,1},\cdots,X_{n,\tau_n})$ and post-change segment $(X_{n,\tau_n+1},\cdots,X_{n,n})$ are  separately exchangeable (within themselves) and mutually independent (of each other). The estimator $\widehat{\tau}_n \equiv \widehat{\tau}_n(\X_n)$ is \emph{consistent} for $\tau_n$ if $\frac{\widehat{\tau}_n}{\tau_n} \convP 1$ as $n \to \infty$. 

\subsection{Related Works}
\label{Related work}

\paragraph{Changepoint detection and localization.}
Classical changepoint analysis, surveyed in \cite{truong2020selective}, predominantly relies on parametric, likelihood-based formulations. While effective under correctly specified distributions, these methods typically yield point estimates without finite-sample confidence statements. Although \cite{saha2025post} recently constructed confidence sets following sequential detection, their approach requires parametric assumptions and restricted distribution classes. Similarly, rapid detection procedures like CUSUM \citep{page1955test} and conformal test martingales \citep{vovk2003testing} are designed for online signaling rather than offline localization confidence sets. \cite{dandapanthula2025conformal} introduced the first nonparametric algorithm with finite-sample coverage guarantees, yet it lacks further theoretical characterization.

\paragraph{Nonparametric confidence sets for changepoint.}
Existing methods for changepoint confidence sets generally rely on asymptotic validity or specific model structures, such as mean shifts. \cite{cho2022bootstrap} and \cite{fotopoulos2010exact} provide asymptotic solutions for univariate changes, while \cite{frick2014multiscale} and \cite{xu2024change} address regression settings under restrictive conditions. Although \cite{verzelen2023optimal} offer finite-sample guarantees for mean changes, their reliance on unknown constants prevents practical implementation. In contrast, \cite{dandapanthula2025conformal} propose the MCP algorithm for assumption-free, finite-sample valid confidence sets; this work provides the associated theoretical guarantees.

\paragraph{Conformal $p$-value.}
Introduced by \cite{vovk2005algorithmic, shafer2008tutorial} and reviewed in \cite{angelopoulos2024theoretical}, conformal $p$-values are now central to two-sample tests \citep{hu2024two} and conditional testing \citep{wu2024conditional}, outlier detection \citep{bates2023testing, zhang2022automs}, and data selection \citep{jin2023selection, wu2024optimal}. \cite{vovk2022testing} utilized randomized conformal $p$-values to test exchangeability via martingales; however, distributional properties were established only under the null hypothesis (exchangeability), preventing a theoretical characterization of the test's power against non-exchangeable alternatives.

\paragraph{Testing exchangeability.}
Beyond conformal martingales \cite{vovk2022testing}, the ``testing by betting'' framework \cite{shafer2021testing, ramdas2022testing} has enabled new exchangeability tests. Notably, \cite{saha2024testing} derived consistent tests for binary Markovian and Gaussian AR(1) processes. However, these target specific parametric dependencies. We propose a distribution-free test consistent against the broader alternative of general distributional changes.

\paragraph{Changepoint \emph{Detection} using Conformal methods.}
Conformal changepoint methods have historically prioritized detection over localization. \cite{vovk2003testing} pioneered online exchangeability testing, which \cite{vovk2021testing, vovk2021retrain} refined into conformal martingales and \cite{shin2022detectors} generalized via the e-detector framework. While \cite{volkhonskiy2017inductive} developed computationally efficient inductive variants and \cite{nouretdinov2021conformal} established detection efficiency, these works do not address confidence set construction for the changepoint location.

\subsection{Our Contributions}
\begin{itemize}
    \item Existing changepoint localization methods rely on restrictive parametric assumptions, asymptotic validity, or specific change types (e.g., mean shifts). While \cite{dandapanthula2025conformal} introduced the first distribution-free algorithm with finite-sample validity, they provided no analysis of set size, leaving open the possibility of trivial, uninformative sets (e.g., $[n-1]$). We provide the first theoretical guarantees on the size of these confidence sets, establishing the algorithm as a practically implementable method for distribution-free localization with finite-sample validity.
    \item  We propose a novel distribution-free changepoint estimator based on conformal $p$-values, and show that our estimator is consistent without any parametric assumption on the pre- and post-change distributions. Furthermore, we also find the convergence rate of the estimated changepoint to the original changepoint under very mild assumptions.
    \item Conformal $p$-values are central to many recent advances in hypothesis testing \citep{vovk2003testing, wu2024conditional} and outlier detection \citep{bates2023testing}, yet their distributional properties were previously understood only under exchangeability \citep{vovk2003testing}. We derive novel finite-sample and asymptotic properties of conformal $p$-values under distribution shifts, characterizing their precise deviation from uniformity. These results provide a rigorous foundation for applying conformal methods in non-exchangeable, distribution-changing settings.
     \item We propose a distribution-free test to assess whether a specific time point is a changepoint, proving its consistency against local alternatives under mild assumptions specified in \Cref{Properties of MCP}. Additionally, we construct a distribution-free test of exchangeability that is consistent against distribution-change alternatives. Together, these  establish a unified, assumption-free framework for changepoint detection, localization, and testing with finite-sample validity.
    \item The performance of conformal methods depends critically on the score function. While the likelihood ratio is a theoretically optimal score function \citep{dandapanthula2025conformal}, it requires knowledge of unknown pre- and post-change distributions. We propose a computationally implementable \textit{nearly optimal} score function and prove that the power difference between the conformal $p$-values using our score function and the oracle converges to zero asymptotically. This yields a practical method that effectively maximizes the power of the conformal $p$-value in the distribution change setup.
\end{itemize}

\paragraph{Structure of the paper.}We start with developing properties of conformal $p$-values in the distributional change setting in \Cref{Properties of conformal pvalue}. Theoretical guarantees for both our changepoint detection method and the localization procedure of \cite{dandapanthula2025conformal} are presented in \Cref{Properties of MCP}. We construct distribution-free consistent tests for the changepoint and that of exchangeability in \Cref{sec: distribution-free tests for changepoint and exchangeability}. In \Cref{Nearly optimal score fucntion}, we construct \textit{nearly optimal} score functions and establish their asymptotic optimality. \Cref{Appendix disMCP} contains a detailed discussion and motivation of the Algorithm. A series of simulation studies designed to corroborate the performance of the algorithms and to illustrate the theoretical properties established is available in \Cref{simulation studies}. The form of optimal score functions for different kinds of distribution changes are explicitly derived in \Cref{Appendix A}. The form of our algorithm for growing samples in the form of a triangular array is stated in \Cref{sec: Appendix B}. Finally, proofs of all the results of this paper are presented in \Cref{Appendix: Proofs}.

\section{Properties of Conformal $p$-value under distribution change}
\label{Properties of conformal pvalue}
For random variables $X_1,\cdots,X_n$ and a score function $s: \R \to \R$, the conformal $p$-value is defined as $p_n^*[s] = \frac{1}{n}\sum_{i=1}^n \mathbf{1}\big(s(X_i) \leq s(X_n)\big).$ Without loss of generality, we consider the identity score function $s(x) = x$ and denote $p^*_n[s]$ simply as $p^*_n$, since one can equivalently analyze the transformed data $s(X_1), \dots, s(X_n)$. A randomized version of the above is defined as
\begin{equation}
    \label{def:conf pval}
    p_n(X_1,\cdots,X_n;\theta_n) :=  \frac{\sum_{i=1}^n \mathbf{1}\big(X_i < X_n\big) + \theta_n \sum_{i=1}^n \mathbf{1}\big(X_i = X_n\big)}{n},
\end{equation}
where $\theta_n \sim U(0,1)$, independent of $(X_1,\cdots,X_n)$. \cite{vovk2003testing,vovk2005algorithmic} proved the following.
\begin{fact}
\label{Vovk Fundamental theorem}
    Let $X_1,\cdots,X_n$ be exchangeable. For $k \in [n]$, define  $p_k := p_k(X_1,\cdots,X_k;\theta_k)$
    where $\theta_1,\cdots,\theta_n \iid U(0,1);(\theta_1,\cdots,\theta_n) \perp (X_1,\cdots,X_n)$. Then $p_1,\cdots,p_n \iid U(0,1)$. 
\end{fact}
In this paper, we study the properties of $p_n$ under the distributional change i.e.\ when
for some $\tau_n \in [n-1], (X_1,\cdots,X_{\tau_n}) \sim \P_0$ and $(X_{\tau_n+1},\cdots,X_n) \sim \P_1$ where $(X_1,\cdots,X_{\tau_n}) \perp (X_{\tau_n+1},\cdots,X_n)$ and $\P_0,\P_1$ are exchangeable. To derive theoretical results, we assume in addition that the data points $X_1,\cdots,X_n$ are independent. Now, since under exchangeability, the $p$-value $p_n$ follows the $U(0,1)$ distribution, when the assumption of exchangeability fails, the distribution of $p_n$ is expected to deviate from the $U(0,1)$ distribution. In the following theorem, we quantify the deviation. 
\begin{theorem}
\label{th2}
Suppose $X_1,\cdots,X_\tau \iid R$ and $X_{\tau+1},\cdots,X_n \iid Q$  for some $\tau \in [n-1]$, where $Q \ll R$, $\frac{dQ}{dR}(x) < \infty$ for all $x \in \R$ and $(X_1,\cdots,X_\tau) \perp (X_{\tau+1},\cdots,X_n)$. Let $\lambda_n \sim U(0,1)$ be independent of $(X_1,\cdots,X_n)$. Define $p_n := p_n(X_1,\cdots,X_n;\lambda_n)$ follwing (\ref{def:conf pval}).
Let \( F_{p_n} \) be the conditional CDF of \( p_n \) given \( (X_{(1)}, \ldots, X_{(n)}) \). Then $$\sup_{t \in [0,1]} \left| F_{p_n}(t) - t \right| 
= \max_{1 \leqslant i \leqslant n} \left| \frac{ \sum_{j \leqslant i} \frac{dQ}{dR}(X_{(j)}) }{ \sum_{j \in [n]} \frac{dQ}{dR}(X_{(j)}) } - \frac{i}{n} \right|.$$   
\end{theorem}
Note that the conclusion of the above theorem does not depend on the position of the changepoint $\tau$, and the result does not particularly assume that $Q,R$ are different. Thus, when $Q = R$, i.e., the entire data is exchangeable, then $\sup_{t \in \R}|F_{p_n}(t) - F(t)| = 0$, implying that $p_n \sim U(0,1)$. So in this sense, \Cref{Vovk Fundamental theorem} is a special case of the above theorem. Also note that the above is a finite sample result, in the sense that it holds for all sample sizes $n \in \N$. 

Now we start discussing asymptotic properties of the conformal $p$-value under the distribution change setup. To accommodate growing sample sizes, we consider a triangular array $\{X_{n,j}\}_{n \in \N, 1 \leq j \leq n}$, where the $n$-th row represents the sample of size $n$ satisfying the following assumption.
\begin{assumption}[Triangular Array with Changepoint]
    \label{assn:trig array}
    For a fixed $c \in (0,1)$, $\tau_n = \lfloor cn \rfloor$, assume that $X_{n,j} \overset{\text{iid}}{\sim} R \text{ for } j \leq \tau_n, \text{ and }  X_{n,j} \overset{\text{iid}}{\sim} Q \text{ for } j > \tau_n,$ where $Q \ll R, \frac{dQ}{dR}(x) < \infty$ for all $x \in \R$ and $(X_{n,1},\cdots,X_{n,\tau_n}) \perp (X_{n,\tau_n+1},\cdots,X_{n,n})$ for all $n$.
\end{assumption}

So $\tau_n$ signifies the changepoint for the sample of size $n$ which does not lie in the extreme ends of the data stream. The next result (\Cref{Th:th3}) is an asymptotic version of \Cref{th2} in the above setup. \Cref{th2} quantifies the deviation of the post-change $p$-value $p_n$ from the standard uniform distribution for a finite sample of fixed size, while \Cref{Th:th3} studies the asymptotic deviation of the post-change $p$-value from the standard uniform distribution.  To get the next result, we need the following assumption on $c$ and the distributions $R,Q$. 
\begin{assumption}
\label{assn1}
    Assume there exists $\delta \in (0,1)$ such that: 
    
   (i) $\delta$th quantile of $cR + (1-c)Q$ is unique (denote it $q_{\delta, c}$), and $R$ is continuous at $q_{\delta, c}$. 


   (ii) $\E_{Z \sim Q} \left( \frac{dQ}{dR}(Z) \right) < \infty$ and $\mu \neq 1$, where
    \begin{align*}
\mu &= \frac{\frac{c}{\delta} R(q_{\delta, c}) \cdot \E_{Z \sim R} \left( \frac{dQ}{dR}(Z) \right) 
+ \left(1 - \frac{c}{\delta} R(q_{\delta, c}) \right) \cdot \E_{Z \sim Q} \left( \frac{dQ}{dR}(Z) \right)}
{c \cdot \E_{Z \sim R} \left[ \frac{dQ}{dR}(Z) \right] + (1 - c) \cdot \E_{Z \sim Q} \left[ \frac{dQ}{dR}(Z) \right]}.
\end{align*}

\end{assumption}
There are many natural ways to satisfy \Cref{assn1}. For example, if one assumes $R,Q$ have continuous densities $r,q$ with respect to the Lebesgue measure; $r,q$ have identical support, and the common support $\{x: r(x) > 0\}$ is connected, then \Cref{assn1} holds. 
As an example, if \( R \) is  of \( N(\mu_1, \sigma_1^2) \) and \( Q \) is  \( N(\mu_2, \sigma_2^2) \) with \( (\mu_1, \sigma_1^2) \neq (\mu_2, \sigma_2^2) \), then \Cref{assn1} is satisfied. With the above motivation, we now state the following result. 
\begin{theorem}
\label{Th:th3}
Consider a triangular array $\{X_{n,j}\}_{ n \in \mathbb{N}, 1 \leqslant j \leqslant n}$, and let $\lambda_n \overset{\text{iid}}{\sim} U(0,1)$ with $\{\lambda_n\}_{ n \in \N}$ independent of the array. Define $p_n := p_n(X_{n,1},\cdots,X_{n,n};\lambda_n)$ following (\ref{def:conf pval}). Let $F_{p_n}$ be the conditional CDF of $p_n$ given $\{X_{n,i}\}_{i=1}^n$. Then, under \Cref{assn:trig array},~\ref{assn1},
\begin{equation}
\label{eq:main_theorem}
\liminf_{n \to \infty} \sup_{t \in [0,1]} \left| F_{p_n}(t) - t \right| > 0 \quad \text{a.s.}
\end{equation}
\end{theorem}
From the above theorem, it can be seen that the asymptotic distribution of the post-change $p$-values differs from the standard uniform distribution. The proof of the result proceeds by first computing the KS distance $\sup_{t \in [0,1]} |F_{p_n}(t) - t|$ explicitly using \Cref{th2} and then establishing a lower bound on that by analyzing the corresponding likelihood ratio process at a specific rank $i=\delta n$. We decompose the sum $\sum_{j \leq \delta n} \frac{dQ}{dR}(X_{n,(j)})$ based on the latent source of the sample observations (pre or post-change) and then rigorously derive the asymptotic proportion of samples from $R$ residing in the bottom $\delta$-ranks of the entire sample. The primary technical novelty of the above and next results lies in our asymptotic analysis of partial sums of the likelihood ratio order statistic of the sample as standard methods for rank-based statistics can not be used in the setup for the distribution-change. 

So far, the above results focus on the properties of a single post-change $p$-value. However, it is quite imperative to understand the joint asymptotic behavior of the multiple post-change $p$-values. Aiming towards that, as a first step, we investigate how abruptly the distribution of the post-change $p$-values changes after the change occurs; So we obtain an upper bound of the rate in the next theorem. In view of that, we need the following assumptions on $c,R,Q$. 
\begin{assumption}
\label{assn2}
$c \in (0,1)$ is such that the following holds simultaneously:
\begin{itemize}
    \item[(i)] For all $\delta \in (0,1)$, the $\delta$-th quantile of the mixture distribution $cR + (1-c)Q$ is unique.
    
    \item[(ii)] $R$, $Q$ are Lebesgue absolutely continuous and have continuous densities $r$, $q$ such that
     $$\mathbb{E}_{Z \sim Q} \left[ \left( \frac{dQ}{dR}(Z) \right)^2 \right] < \infty 
    \quad \text{and} \quad 
    \mathbb{E}_{Z \sim R} \left[ \left( \frac{dQ}{dR}(Z) \right)^2 \right] < \infty.$$
\end{itemize}
\end{assumption}
There are many natural ways to satisfy \Cref{assn2}. For example, if $r,q$ have the same support and the common support $\{x: r(x) > 0\}$ is connected, then $(i)$ is satisfied because the CDF of the mixture distribution $cR + (1-c)Q$ is strictly increasing on the common support and $(ii)$ is a very standard second moment assumption assumed for technical purposes. As before, if $Q,R$ are Gaussian, 
then \Cref{assn2} is satisfied.
 With this assumption, we now state the following theorem providing an upper bound on the rate of the distributional change of the post-change $p$-values. 
\begin{theorem}
\label{th: Th diff pval}
Let $\{X_{n,j}\}_{n \in \mathbb{N}, 1 \le j \le n}$, $p_n, F_{p_n}$ be as in \Cref{Th:th3}. Then, under \Cref{assn:trig array},~\ref{assn2},
\begin{equation}
\sup_{\epsilon \in \mathbb{R}} \left| F_{p_n}(\epsilon) - F_{p_{n+1}}(\epsilon) \right| = O_p\left( \frac{1}{\sqrt{n}} \right). \label{eq:thm}
\end{equation}
\end{theorem}
The proof of the above result proceeds by exploiting the piecewise linear structure of $F_{p_n}$ obtained from \Cref{th2}. This geometry allows us to reduce the supremum over $\epsilon \in [0,1]$ to a discrete comparison at grid points of the form $i/n$. At these points, the difference decomposes into two terms: a linear interpolation error and the incremental fluctuation of normalized partial sums, which are further analyzed to show the stochastic boundedness of $\sqrt{n} \sup_{\epsilon \in [0,1]} |F_{p_n}(\epsilon) - F_{p_{n+1}}(\epsilon)|$.

We finally establish our principal result on the conformal $p$-value, which serves as the foundation for the subsequent analysis of the changepoint algorithm. It studies the empirical property of the conformal $p-$values $p_{n,1},\cdots,p_{n,n}$ computed on a sample $\{X_{n,j}\}_{1 \leq j \leq n}$ of size $n$. It is easy to see using \Cref{Vovk Fundamental theorem}, for a sample $\{X_{n,j}\}_{1 \leq j \leq n}$ of size $n$, at which the distribution change occurs at timepoint $\tau_n$, the pre-change $p$-values satisfy $p_{n,1},\cdots,p_{n,\tau_n} \iid U(0,1)$. However, the post-change $p$-values $p_{n,\tau_n+1},\cdots,p_{n,n}$ are possibly not even independent, and their joint distribution is also not easily tractable. As a result, conventional tools fall short in describing the behavior of the empirical distribution for all the $p$-values together. This limitation largely motivates the novelty and significance of the result we present below.
\begin{theorem}
\label{th: Main theorem}
    Consider a triangular array $\{X_{n,j}\}_{ n \in \mathbb{N}, 1 \leqslant j \leqslant n}$ under  \Cref{assn:trig array},~\ref{assn1},~\ref{assn2}. Let $\lambda_{n,j} \overset{iid}{\sim} U(0,1)$ with $\{\lambda_{n,j}\}_{ n \in \mathbb{N}, 1 \leq j \leq n}$ independent of the array. Define $p_{n,j} := p_j(X_{n,1},\cdots,X_{n,j};\lambda_{n,j})$ following (\ref{def:conf pval}). Let $\widehat{F}_{n,j}$ be the empirical CDF of $p_{n,1}, \ldots, p_{n,j}$.
Then, for all $\Tilde{c} > 0$,
\begin{equation}
\liminf_{n \to \infty} \sup_{t \in [0,1]} \left| \widehat{F}_{n, \lfloor cn + \Tilde{c}n^{1/4} \rfloor}(t) - t \right| > 0 \quad \text{a.s.}.
\label{eq:main}
\end{equation}
\end{theorem} 
The proof of the above result proceeds by lower-bounding the KS distance $\sup_{t \in [0,1]} \left| \widehat{F}_{n, \lfloor cn + \Tilde{c}n^{1/4} \rfloor}(t) - t \right|$ by the sum of the deviations of both the terms from the average conditional CDF of the $p$-values, $\frac{1}{\lfloor cn + n^{1/4} \rfloor} \sum_{k=1}^{\lfloor cn + n^{1/4} \rfloor} \P_{\mathcal{G}_{n,k}}(p_{n,k} \leq t)$, where $\P_{\mathcal{G}_{n,j}}(\cdot)$ denote the conditional probability conditioned on $\mathcal{G}_{n,j} = \sigma\left( \{ X_{n,k} \}_{1 \leq k \leq j} \right)$, and then showing on of the deviations are asymptotically negligible while the other has an almost surely positive limit inferior. 

It is to be noted that for a sample $\{X_{n,j}\}_{1 \leq j \leq n}$ of size $n$, the pre-change $p$-values are independently and identically distributed as the standard uniform distribution, while in contrast, the post-change $p$-values significantly deviate from that. Consequently, if one intends to apply the Kolmogorov–Smirnov (KS) test to assess the exchangeability of the sample, the preceding result specifies the exact number of $p$-values that must be taken into account in order to obtain a consistent test. Thus, the above result has notable significance in the context of testing exchangeability details of which are discussed in \Cref{sec: distribution-free tests for changepoint and exchangeability}. In addition, the above insight serves as the foundation for the theoretical guarantees of the changepoint algorithm, which will be developed and discussed in detail in the following section.

\section{The Matrix of Conformal $p$-values (MCP) Algorithm for changepoint detection and localization: Description and Theoretical Properties}
\label{Properties of MCP}
As formulated in \Cref{subsec: probolem formulation}, our goal is to estimate the changepoint $\tau_n$ from data. We begin by defining the score function, a foundational component of the conformal framework. Throughout, let $[\mathcal{X}^m]$ denote the set of unordered \emph{multisets} (bags) of $m$ observations, denoted by $[Y_1, \ldots, Y_m]$.

\begin{definition}[Score Functions]
    A \textbf{family of score functions} is a collection $\big( (s_{r,t})_{r=1}^{n-1}\big)_{t=1}^n$ of functions $s_{r,t} : \mathcal{X} \times [\mathcal{X}^r] \times \mathcal{X}^{n-t} \to \mathbb{R}$.
\end{definition}
Intuitively, the score function projects data into one dimension to enhance separability between pre- and post-change distributions. The second argument processes a multiset of observations exchangeably, ensuring the score depends only on the set of pre-change data, not its order. Conversely, the third argument processes data non-exchangeably, allowing flexibility to learn from the sequence structure. The MCP changepoint detection algorithm is formally described in Algorithm~\ref{alg: CONCH}.

\begin{algorithm2e}[!t]
\small
\caption{MCP Algorithm for Changepoint Detection}
\label{alg: CONCH}
\PrintSemicolon 

\KwIn{$(X_{t})_{t=1}^n$ (dataset), $\big((s_{r,t}^{(0)})_{r =1}^{n-1}\big)_{t=1}^n$ and $\big((s_{r,t}^{(1)})_{r =1}^{n-1}\big)_{t=1}^n$ (left and right score function families)}
\KwOut{$\widehat{\tau}_n$ (estimate of changepoint $\tau_n$)}

\For{$t \in [n-1]$}{
    \For{$r \in (1, \ldots, t)$}{
        \For{$j \in (1, \ldots, r)$}{
            $\kappa_{r,j}^{(t)} \leftarrow s_{r,t}^{(0)}(X_j ; [X_1, \ldots, X_r], (X_{t+1}, \ldots, X_n))$\;
        }
        \textbf{end}\;
        Draw $\theta_r^{(t)} \sim \text{Unif}(0,1)$\;
        $p_r^{(t)} \leftarrow \frac{1}{r} \sum_{j=1}^{r} \left( \mathbf{1}_{\kappa_{r,j}^{(t)} > \kappa_{r,r}^{(t)}} + \theta_r^{(t)} \cdot \mathbf{1}_{\kappa_{r,j}^{(t)} = \kappa_{r,r}^{(t)}} \right)$\;
    }
    \textbf{end}\;

    \For{$r \in (n, \ldots, t+1)$}{
        \For{$j \in (r+1, \ldots, n)$}{
            $\kappa_{r,j}^{(t)} \leftarrow s_{n-r,n-t}^{(1)}(X_j ; [X_{r+1}, \ldots, X_n], (X_1, \ldots, X_t))$\;
        }
        \textbf{end}\;
        Draw $\theta_r^{(t)} \sim \text{Unif}(0,1)$\;
        $p_r^{(t)} \leftarrow \frac{1}{n - r + 1} \sum_{j=r}^{n} \left( \mathbf{1}_{\kappa_{r,j}^{(t)} > \kappa_{r,r}^{(t)}} + \theta_r^{(t)} \cdot \mathbf{1}_{\kappa_{r,j}^{(t)} = \kappa_{r,r}^{(t)}} \right)$\;
    }
    \textbf{end}\;

    $\widehat{F}^{(0)}_{t}(z) := \frac{1}{t} \sum_{r=1}^{t} \mathbf{1}_{p_{r}^{(t)} \leq z}$; \quad
    $\widehat{F}^{(1)}_{t}(z) := \frac{1}{n - t} \sum_{r=t+1}^{n} \mathbf{1}_{p_{r}^{(t)} \leq z}$\;
    
    $W_{t}^{(0)} \leftarrow \sqrt{t} \cdot \text{KS}(\widehat{F}^{(0)}_{t}, u)$; \quad
    $W_{t}^{(1)} \leftarrow \sqrt{n - t} \cdot \text{KS}(\widehat{F}^{(1)}_{t}, u)$\;

    $p_{t}^\text{left} = 1 - F_t(W_{t}^{(0)})$; \quad $p_{t}^\text{right} = 1 - F_{n-t}(W_{t}^{(1)})$\;

    ($F_n$ is the CDF of $\sqrt{n}\cdot \text{KS}(\widehat{F}_n,u)$, $\widehat{F}_n$ is the empirical CDF of $n$ iid observations from $U(0,1)$)\;

    $p_t^{\texttt{CONF}} = \min \{2p_t^{\text{left}}, 2p_t^{\text{right}}, 1\}$\;
}
\textbf{end}\;

$\widehat{\tau}_n \leftarrow \arg\max_{t \in [n-1]} p_t^{\texttt{CONF}}$\;
\Return $\widehat{\tau}_n$\;
\end{algorithm2e}
As discussed in \cite{dandapanthula2025conformal}, when $X_1,\cdots,X_\tau \iid R$ and $X_{\tau+1},\cdots,X_n \iid Q,$ one can find the finite-sample valid $(1-\alpha)$ level confidence set for $\tau_n$ as,
\begin{equation}
\label{eq:CONCH CI}
    C^\texttt{MCP}_{n,1-\alpha} = \{t \in [n-1] : p^{\texttt{CONF}}_{t} > \alpha \},
\end{equation}
where $p^{\texttt{CONF}}_{t}$ is as in Algorithm~\ref{alg: CONCH} (Line 20). \Cref{fig:nprmal pval} contains an illustration of the $p$-values produced by the MCP algorithm, showing that in the presence of a distribution
change, the $p$-values exhibit a sharp spike-type pattern around the true changepoint. The motivation and related discussions regarding the algorithm can be found in \Cref{Appendix disMCP}. A visualization of the characteristics of the $p$-values is in \Cref{simulation studies}.
In particular, when there is a triangular array of random variables $\{X_{n,j}\}_{n \in \N, 1 \leq j \leq n}$ satisfying $X_{n,1},\cdots,X_{n,\tau_n} \iid R$, $ X_{n,\tau_n+1},\cdots,X_{n,n} \iid Q$, $ (X_{n,1},\cdots,X_{n,\tau_n}) \perp (X_{n,\tau_n+1},\cdots,X_{n,n})$, the form of the MCP algorithm (Algorithm~\ref{alg: CONCH}) for applying on data points in one row $\{X_{n,j}\}_{1 \leq j \leq n}$, is discussed in \Cref{sec: Appendix B} as Algorithm~\ref{alg: CONCH traiangular} for notational convenience. That form of the algorithm will be referred to during the discussions related to the asymptotic properties of the algorithm.  

Following \cite{dandapanthula2025conformal}, although $C^\texttt{MCP}_{n,1-\alpha}$ has finite sample validity, analyzing its size is of crucial importance for practical applications, as one can anyway consider the set $[n-1]$ for a complete coverage confidence set. So in that view, we now discuss the theoretical properties for the point estimator and confidence set produced by the MCP algorithm. In the next result, we show that for any miscoverage level $\alpha \in (0,1)$, if one applies the MCP algorithm to obtain the $(1-\alpha)$ confidence set $C_{n,1 - \alpha}^\texttt{MCP}$ from the sample of size $n$, and $l_{n,1-\alpha}$ denotes its size, i.e. $l_{n,1-\alpha} = \sum_{j = 1}^n \mathbf{1}(j \in C_{n,1 - \alpha}^\texttt{MCP})$ then $\frac{l_{n,1-\alpha}}{n} \convAS 0$ as $n \to \infty$ under weak assumptions.
\begin{theorem}
\label{th: rel len CONCH}
   Let $\{X_{n,j}\}_{ n \in \mathbb{N}, 1 \leqslant j \leqslant n}$ be such that \Cref{assn:trig array},\ref{assn1},\ref{assn2} hold. For $\alpha \in (0,1)$, let $\widehat{\tau}_n$ and $C^\texttt{MCP}_{n, 1- \alpha}$ be the point estimate and $(1 - \alpha)$ confidence set respectively for $\tau_n$ using Algorithm~\ref{alg: CONCH traiangular} applied on $\X_n := \{X_{n,j}\}_{1 \leq j \leq n}$. Let $l_{n,1-\alpha}$ be the size of the confidence set $C^\texttt{MCP}_{n, 1- \alpha}$. Then, the relative size of the confidence set shrinks to zero:  $\frac{l_{n,1-\alpha}}{n} \convAS 0,  \text{ as } n \to \infty.$
   In addition, $\mathbf{1}\bigg(\widehat{\tau}_n \in \big[\tau_n - n^{1/4}, \tau_n + n^{1/4}\big]\bigg) \convAS 1,$ yielding ratio-consistency:  $\frac{\widehat{\tau}_n}{\tau_n} \convAS 1,  \text{ as }  n \to \infty.$
\end{theorem}
The proof of the above is done using a careful analysis of the power of the Kolmogorov–Smirnov test (at line 18 of \Cref{alg: CONCH}) using \Cref{th: Main theorem}. 
The result provides us with several interesting insights. Although the size of the confidence set is expected to grow with the sample size, this result in particular guarantees that the \emph{relative} size shrinks: if one shrinks the whole time interval $[1,n]$ to $[0,1]$, in case the shrunk confidence set $C_{n,1 - \alpha}^\texttt{MCP}$ almost surely converges to the changepoint itself. Also, using the dominated convergence theorem, it is immediate that $\frac{1}{n}\sum_{j=1}^n \P(j \in C_{n,1-\alpha}^\texttt{MCP}) \to 0$. The second result shows that the point estimator of $\widehat{\tau}_n$ of $\tau_n$ is consistent at a rate $n^{1/4}$ i.e.\ $\widehat{\tau}_n - \tau_n = o_P(n^{\frac{1}{4} + \epsilon})$ for all $\epsilon > 0$. An empirical justification for \Cref{th: rel len CONCH} is provided in \Cref{simulation studies}. From \Cref{fig:rel_length_CONCH}, it is evident that the relative length of the confidence set
produced by the MCP Algorithm tends to 0 as the sample size $n$ grows. In addition, \Cref{fig:consistency} validates the consistency of $\widehat{\tau}_n$ by displaying the ratio $\widehat{\tau}_n/\tau_n$ for growing sample sizes, in various distribution-shift settings.

\section{Distribution-free tests for changepoint and exchangeability}
\label{sec: distribution-free tests for changepoint and exchangeability}

\paragraph{Test for changepoint.} Using the MCP framework, we construct tests to assess whether a specific time point corresponds to a distribution changepoint. As before, let $\tau_n \in [n-1]$ be the original changepoint for the sample $\{X_{n,j}\}_{1 \leq j \leq n}$ of size $n$ and let $\{t_n\}_{n \in \N}$ and $\{t_n'\}_{n \in \N}$
be two sequences of naturals satisfying $t_n = \lfloor cn \rfloor$ and $t_n' = \lfloor c'n \rfloor$ for all $n$ and fixed $c,c' \in (0,1)$ with $c \neq c'$. Observe that there can be at most finitely many values of $n$, with $\lfloor cn \rfloor = \lfloor c'n \rfloor$ since $\lim_{n \to \infty}|c- c'|n = +\infty$. Since we are interested in asymptotic properties, without loss of generality, we can assume $\lfloor cn \rfloor \neq \lfloor c'n \rfloor$ for all values of $n$, since otherwise we can take sufficiently large values of $n$ in our consideration and continue the further discussion.  Let $\mathcal{H}_0^{(n)}$ and $\mathcal{H}_1^{(n)}$ be the hypotheses 
\begin{equation}
\label{eq:hypo}
\begin{split}
    \mathcal{H}_0^{(n)} : \tau_n = t_n, \quad 
    \mathcal{H}_1^{(n)} : \tau_n = t_n'.  
\end{split}
\end{equation}
Then to test $\mathcal{H}_0^{(n)}$ based on the sample $\mathbf{X}_n := \{X_{n,j}\}_{1 \leq j \leq n}$, for a fixed $\alpha \in (0,1)$, consider the test $\phi_{n,\alpha}$ obtained by inverting the confidence set $C_{n,1-\alpha}^\texttt{MCP}$ i.e. 
\begin{equation}
\label{eq: inverted test}
    \phi_{n,\alpha}(\mathbf{X}_n) = \mathbf{1}\big( t_n \notin C_{n,1 - \alpha}^\texttt{MCP} \big) = \mathbf{1}\big( p_{n,t_n}^\texttt{CONF} \leq \alpha \big).
\end{equation}
By \Cref{Vovk Fundamental theorem}, we have $\E_{\mathcal{H}_{0}^{(n)}}\big(\phi_{n,\alpha}(\mathbf{X}_n)\big) \leq \alpha$. So without any further distributional assumption on $R,Q$ we can conclude, $\phi_{n,\alpha}$ is a size $\alpha$ test for $\mathcal{H}_{0}^{(n)}$. In addition, the sequence of tests $\{\phi_{n,\alpha}\}_{n}$ is consistent against the sequence of alternatives $\mathcal{H}_1^{(n)}$, meaning that $\E_{\mathcal{H}_1^{(n)}}\big(\phi_{n,\alpha}(\mathbf{X}_n)\big) \to 1$ as $n \to \infty$, as guaranteed by the following result. 

\begin{proposition}[Consistency against a local alternative]
    Let $\{X_{n,j}\}_{ n \in \mathbb{N}, 1 \leqslant j \leqslant n}$ be such that \Cref{assn:trig array},\ref{assn1},\ref{assn2} hold and $\{\mathcal{H}_0^{(n)}\}_n$, $\{\mathcal{H}_1^{(n)}\}_n$ be as in (\ref{eq:hypo}). Then the sequence of tests $\{\phi_{n,\alpha}\}_n$ defined in (\ref{eq: inverted test}) satisfy,
   (i) $\phi_{n,\alpha}$ has level $\alpha$ for $\mathcal{H}_0^{(n)}$ for all $n \in \N$.  
   (ii) $\{\phi_{n,\alpha}\}_n$ is consistent against  $\{\mathcal{H}_1^{(n)}\}_n$. 
    
\end{proposition}
Thus, for a large value of $n$, if we have the sample $\mathbf{X}_n := \{X_{n,j}\}_{1 \leq j \leq n}$ having changepoint $\tau_n$, and we want to test $\mathcal{H}_0^{(n)} : \tau_n = cn$ vs $\mathcal{H}_1^{(n)} : \tau_n = c'n$, $\phi_{n,\alpha}$ is a size $\alpha$ test with power close to $1$.

\paragraph{Test for exchangeability.}

Using the conformal $p$-value, we now propose an offline test for exchangeability. As before, we have the sample $\mathbf{X}_n := \{X_{n,j}\}_{1 \leq j \leq n}$ of size $n$ and we want to test the exchangeability of the sample against the distribution-change alternative i.e.,
\begin{equation}
\label{ex hypo}
\begin{split}
    \Tilde{\mathcal{H}}_0^{(n)} &: X_{n,1},\cdots,X_{n,n}\hspace{0.2cm} \text{are exchangeable.} \\
    \Tilde{\mathcal{H}}_1^{(n)} &: X_{n,1},\cdots,X_{n,\lfloor cn \rfloor} \iid R \hspace{0.2cm} \text{and} \hspace{0.2cm} X_{n,\lfloor cn\rfloor + 1},\cdots, X_{n,n} \iid Q.
\end{split}
\end{equation}
for a fixed $c \in (0,1)$ which does not depend on $n$. We find the conformal $p$-values $p_{n,1}^\texttt{CONF},\cdots,p_{n,n}^\texttt{CONF}$ according to Algorithm~\ref{alg: CONCH traiangular}. From \Cref{Vovk Fundamental theorem}, we have under $\tilde{\mathcal{H}}_0^{(n)}$, $p_{n,1}^\texttt{CONF},\cdots,p_{n,n}^\texttt{CONF} \iid U(0,1)$. So to test $\tilde{\mathcal{H}}_0^{(n)}$, one can apply the Kolmogorov-Smirnov test on the $p$-values as follws; Denote $\widehat{F}_{n,j}(t) := \frac{1}{j}\sum_{i=1}^j\mathbf{1}\big( p_{n,j}^\texttt{CONF} \leq t \big)$, $W_{n,j} := \sqrt{j}\text{KS}(\widehat{F}_{n,j},u)$, where for a CDF $G$, $KS(G,u) := \sup_{t \in \R}|G(t) - F(t)|,$ where $F$ is the CDF of $U(0,1)$ random variable. Observe the distribution of $W_{n,j}$ is known under $\tilde{\mathcal{H}}_0^{(n)}$, and hence one can find $q_{n,j}^{1-\alpha} := (1-\alpha)$ th quantile of $W_{n,j}$ under $\tilde{\mathcal{H}}_0^{(n)}$. Using Donsker's theorem (\cite{donsker1951invariance}), as $n \to \infty$, $q_{n,\lfloor cn + n^{1/4}\rfloor}^{1 - \alpha} \to q_{1-\alpha}$, where $q_{1-\alpha}$ is the $(1-\alpha)$th quantile of $\sup_{z \in [0,1]}|\phi_z|, \phi_z = B_z - zB_1$ with $(B_z)_{z \geq 0}$ being standard brownian motion on $\R$. Thus to test $\tilde{\mathcal{H}}_0^{(n)}$, against $\tilde{\mathcal{H}}_1^{(n)}$, we consider the test
\begin{equation}
\label{consistent test excha}
    \tilde{\phi}_{n,\alpha}(\mathbf{X}_n) := \mathbf{1}\big( W_{n,\lfloor cn+n^{1/4}\rfloor} > q_{n,\lfloor cn+n^{1/4}\rfloor}^{1-\alpha}\big).
\end{equation}
 Observe that, from \Cref{th: Main theorem}, we have under $\tilde{\mathcal{H}}_1^{(n)}$,  $W_{n,\lfloor cn+n^{1/4}\rfloor} \convAS \infty$, thus $\{\tilde{\phi}_{n,\alpha}\}_n$ is consistent against the sequence of alternatives $\{\tilde{\mathcal{H}}_1^{(n)}\}_n$. The above discussion is be summarized below.
 \begin{proposition}(Distribution-free consistent test for exchangeability against distribution-change alternative)
 \label{prep: test for exchangeability}
     Let $\{X_{n,j}\}_{ n \in \mathbb{N}, 1 \leqslant j \leqslant n}$ be such that \Cref{assn:trig array},\ref{assn1},\ref{assn2} hold and $\Tilde{\mathcal{H}}_0^{(n)}$, $\tilde{\mathcal{H}}_1^{(n)}$ be as in (\ref{ex hypo}). Consider testing $\Tilde{\mathcal{H}}_0^{(n)}$ vs $\tilde{\mathcal{H}}_1^{(n)}$ based on the sample $\X_n := \{X_{n,j}\}_{1 \leq j \leq n}$ of size $n$. For a fixed $\alpha \in (0,1)$ and for each $n \in \N$, let $\tilde{\phi}_{n,\alpha}(\mathbf{X}_n)$ be the test defined in (\ref{consistent test excha}). Then (i) $\tilde{\phi}_{n,\alpha}$ is level $\alpha$ test for $\Tilde{\mathcal{H}}_0^{(n)}$ for all $n \in \N$.  
   (ii) $\{\Tilde{\phi}_{n,\alpha}\}_n$ is consistent against  $\{\Tilde{\mathcal{H}}_1^{(n)}\}_n$.
 \end{proposition}

 \section{Finding nearly optimal score functions for conformal $p$-value}
\label{Nearly optimal score fucntion}

The performance of conformal methods depends critically on the choice of score function. In a beautiful analog to the classical Neyman-Pearson Lemma, in the distribution-change setup, the Conformal Neyman-Pearson Lemma (\Cref{Th:Th1}) \citep{dandapanthula2025conformal} guarantees that the optimal \textit{oracle score function} is in fact the likelihood ratio of the post- and pre-change distributions.
Particular versions of the result when there are particular forms of change in the distribution (e.g., covariate shift, label shift, regression function shift) is discussed in \Cref{Appendix A}. An analogous result of the above, when there are multiple changepoints, is presented in Theorem~\ref{th: CONF NP Lemma mujlti chnagepoint}. 

However, in practice, since the pre and post-change distributions are unknown, one can not compute the oracle score function. So we now discuss how one can obtain a score function which behaves almost like the oracle score function, and we show that the difference of the powers of the conformal $p$-value using our proposed score function and the oracle score function converges to 0. 

As before, let $X_1,\cdots,X_\tau \iid R$ and $X_{\tau+1},\cdots,X_n \iid Q$ where $Q,R$ are absolutely continuous with respect to the Lebesgue measure and have densities $q,r$ respectively. Since the likelihood ratio $s^*(x) = \frac{q(x)}{r(x)}$ corresponds to the most optimal score function, in case $q,r$ are unknown, the most natural choice for that would be the ratio of kernel-density estimates of the densities,
\begin{align}
\label{eq: nearly optimal score function}
    \widehat{\left(\frac{q(x)}{r(x)}\right)} &= 
\frac{\widehat{q}_{X_1, \ldots, X_{\widehat{\tau}_n}}(x)}{\widehat{r}_{X_{\widehat{\tau}_n+1}, \ldots, X_n}(x)},
    \\ \label{eq:KDE est of q}
\widehat{q}_{X_1,\ldots,X_{\widehat{\tau}_n}}(x) 
&:=\frac{1}{\widehat{\tau}_n\,h_{\widehat{\tau}_n}}
\sum_{i=1}^{\widehat{\tau}_n}
K\!\left(\frac{x - X_i}{h_{\widehat{\tau}_n}}\right), \\
\widehat{r}_{X_{\widehat{\tau}_n+1},\ldots,X_{n}}(x)
&:=\frac{1}{(n - \widehat{\tau}_n)\,h_{n - \widehat{\tau}_n}}
\sum_{i=\widehat{\tau}_n+1}^{n}
K\!\left(\frac{x - X_i}{h_{n - \widehat{\tau}_n}}\right),
\end{align}
where $K$ is the kernel function and $h_n$ denotes the bandwidth used for $n$ datapoints. Since we aim for this score function to converge to the optimal score function, we assume that the kernel $K$, bandwidth $h_n$, and densities $q,r$ are such that the kernel density estimates are uniformly strongly consistent; therefore, we assume them to satisfy \Cref{assn: kernel}.
\begin{assumption}[Uniform strong consistency of likelihood ratio]
\label{assn: kernel}
    (1) The kernel $K$ is continuous, bounded variation, with $\lim_{|x| \to \infty}K(x) = 0$.
    (2) Densities $q,r$ are uniformly continuous.
    (3) The bandwidth satisfies $\lim_{n \to \infty} h_n / n^\alpha = 1$ for some $\alpha \in (-1/2, 0).$
\end{assumption}
\begin{algorithm2e}
\small
\caption{Construction of Nearly Optimal Score Function}
\label{alg:optimal-score}
\KwIn{$X_1,\ldots,X_n$ (sample of size $n$)}
\KwOut{Nearly optimal score function $\widehat{s}^*$}
\Begin{
  Compute estimate $\widehat{\tau}_n$ of $\tau$, using Algorithm~\ref{alg: CONCH} using the identity score function\;
  Compute $\displaystyle \widehat{r}(x) :=$ the kernel density estimate of $r(x)$ based on the sample $\{X_1,\ldots,X_{\widehat{\tau}_n}\}$, where the kernel and bandwidth satisfy Assumption~\ref{assn: kernel}\;
  Compute $\displaystyle \hat{q}(x) :=$ the kernel density estimate of $q(x)$ based on the sample $\{X_{\widehat{\tau}_n+1},\ldots,X_n\}$, where the kernel and bandwidth satisfy Assumption~\ref{assn: kernel}\;
}
\Return $\displaystyle \widehat{s}^*(x) \;=\; \frac{\widehat{q}(x)}{\widehat{r}(x)}$\;
\end{algorithm2e}

We summarize the construction of this score function in Algorithm~\ref{alg:optimal-score}. With the above discussion, we now present the following result, which shows that the difference of powers of the conformal $p$-value using our score function and the most optimal score function (the likelihood ratio) converges to $0$ as the sample size tends to $\infty$, guaranteeing the \textit{near optimality} of our proposed score function. We state the result in terms of the triangular array of random variables as before.

\begin{theorem}[Guarantee for the asymptotic optimality of $\widehat{s}^*$]
\label{Th: optimising CONCH}
Let $\{X_{n,j} : n \in \mathbb{N}, 1 \leq j \leq n\}$ be a triangular array of random variables as in \Cref{th: rel len CONCH} satisfying Assumptions~\ref{assn1},~\ref{assn2} and~\ref{assn: kernel}. For a score function $s$, let $p_n[s]$ be the conformal $p$-value defined as 
\begin{equation}
        \label{conf pval traingular}
        p_n[s] = \frac{1}{n} \sum_{i=1}^n \bigg( \mathbf{1}\big(s(X_{n,i}) > s(X_{n,n}) \big) + \theta_n \mathbf{1}\big( s(X_{n,i}) = s(X_{n,n}) \big) \bigg),
    \end{equation}
where $\{\theta_n\}_n \iid U(0,1)$ and $\{\theta_n\}_n$ is independent of $\{X_{n,j}\}_{n \in \N, 1 \leq j \leq n}$. Let $\widehat{s}^*_n$ be the score function obtained using Algorithm~\ref{alg:optimal-score} applied on the sample $\{X_{n,j}\}_{1 \leq j \leq n}$ of size $n$, and $s^*$ be the oracle score function, $s^*(x) = \frac{q(x)}{r(x)}$. Then we have, $\E\big(p_n[s^*]\big) - \E\big(p_n[\widehat{s}^*_n]\big) \to 0 \text{ as } n \to \infty.$
\end{theorem}

Therefore, for samples of large size, one can use the above \textit{nearly optimal} score function, which asymptotically retains almost the same power as the oracle score function and is computationally  implementable without any knowledge of the distributions $Q,R$. 

\section{Conclusion}
\label{sec:conclusion}
We introduced the MCP algorithm for changepoint detection based on conformal \(p\)-values and established a rigorous theoretical foundation for its use in the distribution-change setting. Concretely, we derive finite-sample and asymptotic properties of the conformal \(p\)-value, which provides a rigorous theoretical foundation for many applications of conformal $p$-values, including constructing conformal tests under minimal assumptions. Leveraging these properties, we provide theoretical guarantees for the changepoint localization procedure of \cite{dandapanthula2025conformal} and construct a distribution-free, consistent test to assess whether a particular timepoint is a changepoint or not, showing its consistency against a broad class of alternatives. Thus, our work establishes the MCP Algorithm as a unified framework for the changepoint detection, localization, and testing, providing all the classical desired theoretical guarantees such as finite sample validity, distribution-free property, and consistency.

We also construct a consistent test for exchangeability that retains power against broad alternatives, extending the applicability of conformal $p$-values beyond standard settings. Finally, we propose a computationally tractable, \textit{nearly optimal} score function that asymptotically matches the power of the oracle. Together, these results deliver both principled guarantees and implementable tools, expanding the scope and reliability of conformal methods in distribution-change settings.

\newpage
\appendix
\renewcommand{\thesection}{\Alph{section}} 

\newtheorem{apptheorem}{Theorem}[section] 
\renewcommand{\theapptheorem}{\thesection.\arabic{apptheorem}}

\newtheorem{applemma}[apptheorem]{Lemma} 
\renewcommand{\theapplemma}{\thesection.\arabic{applemma}}

\section{Discussion on the MCP Algorithm for changepoint detection and Localization}
\label{Appendix disMCP}

In this section, we discuss the motivation and further details of the MCP Algorithm. As stated in \Cref{subsec: probolem formulation}, we have the data $X_1,\cdots,X_n$, with a changepoint $\tau_n$. The MCP algorithm (Algorithm~\ref{alg: CONCH}) can be motivated as follows. For each $t \in [n-1]$, consider the null hypothesis,
\[ \mathcal{H}_{0,t} : (X_1,\cdots,X_t) \hspace{0.2cm} \text{exchangeable and} \hspace{0.2cm} (X_{t+1},\cdots,X_n) \hspace{0.2cm} \text{exchangeable}. \]
The algorithm aims to provide a $p$-value
 $p_t$ for this null hypothesis $\mathcal{H}_{0,t}$ as follows. It splits the entire dataset into two parts pivoted at $t$ : $(X_1,\cdots,X_t)$ and $(X_{t+1},\cdots,X_n)$. Now for both of the chunks, each datapoint $X_r$ is first mapped into a one-dimensional statistic $\kappa_{r}^{(t)}$ through a score function (line $4$ and line $11$ of Algorithm~\ref{alg: CONCH}). Since the score function is assumed to be symmetric in all its arguments, under $\mathcal{H}_{0,t}$, exchangeability is still preserved in the chunks $ \kappa^{t,(0)} =(\kappa_1^t,\cdots,\kappa_t^t)$ and $\kappa^{t,(1)} = (\kappa_{t+1}^t,\cdots,\kappa_n^t)$. Now, for each of the two chunks $\kappa^{t,(0)}$ and $\kappa^{t,(1)}$, the conformal $p$-values are sequentially computed (line $7$ and line $14$ in Algorithm $\ref{alg: CONCH}$). Since under $\mathcal{H}_{0,t}$, $(\kappa_1^t,\cdots,\kappa_t^t)$ and $(\kappa_{t+1}^t,\cdots,\kappa_n^t)$ are exchangeable, thus from \Cref{Vovk Fundamental theorem} (\cite{vovk2005algorithmic}) the $p$-values satisfy $p^t_1,\cdots,p^t_t \underset{\mathcal{H}_{0,t}}{\iid} U(0,1)$ and $p^t_{t+1},\cdots,p^t_{n} \underset{\mathcal{H}_{0,t}}{\iid} U(0,1)$.

Thus, to quantify the evidence against the null $\mathcal{H}_{0,t}$, scaled Kolmogorov-Smirnov distance of $(p^t_1,\cdots,p^t_t)$ and $(p^t_{t+1},\cdots,p^t_{n})$ from the $U(0,1)$ distribution is calculated respectively as $W_t^{(0)}$ and $W_t^{(1)}$ (in line 17 of Algorithm~\ref{alg: CONCH}). Since the exact distribution of $W_t^{(0)}$ and $W_t^{(1)}$ is known, therefore $W_t^{(0)}$ and $W_t^{(1)}$ are converted into $p$-values $p_t^\text{left}$ and $p_t^\text{right}$ using the probability integral transform (in line 18 of Algorithm~\ref{alg: CONCH}) which is finally combined into $p_t = p_t^\texttt{CONF}$ (in line 20 of Algorithm~\ref{alg: CONCH}). Thus, from the above discussion, it follows that $p_t$ is a $p$-value for $\mathcal{H}_{0,t}$. Now observe if, for some $\tau_n \in [n-1]$, $(X_1,\cdots,X_{\tau_n}) \sim \P_0$ and $(X_{\tau_n+1},\cdots,X_n) \sim \P_1$ with $\P_0 \neq \P_1$ , then $\mathcal{H}_{0,t}$ is true only for $t = \tau_n$. Thus, in such sense $\widehat{\tau}_n := \argmax_{t \in [n-1]}p_t$ is a natural estimate of the original changepoint $\tau_n$. Furthermore, by Neyman's construction, it also follows that $C^\texttt{MCP}_{n,1-\alpha}$ is a $(1 - \alpha)$ confidence set for the original changepoint $\tau_n$, as stated in the following result. 
\begin{fact}[Finite sample coverage guarantee of $C^\texttt{MCP}_{n,1-\alpha}$ (\cite{dandapanthula2025conformal})]
    For $\tau_n \in [n-1]$, let $(X_1,\cdots,X_{\tau_n}) \sim \P_0$ and $(X_{\tau_n+1},\cdots,X_n) \sim \P_1$ where $\P_0$ and $\P_1$ are exchangeable and $(X_1,\cdots,X_{\tau_n}) \perp (X_{\tau_n+1},\cdots,X_n)$ and let $C^\texttt{MCP}_{n,1-\alpha}$ be as defined in (\ref{eq:CONCH CI}). Then $\P(\tau_n \in C^\texttt{MCP}_{n,1-\alpha}) \geq 1 - \alpha$.
\end{fact}
Since we are focusing on the setting where the distributional change occurs, it is customary to always consider $\tau_n \in [n-1]$. Nonetheless, the behavior of the MCP Algorithm when $\tau_n = n$ (i.e., when there is actually no change in the data) also follows from the above discussion, which is summarized in the following proposition. 
\begin{proposition}
\label{prop:pval exchan}
    Let $X_1,\cdots,X_n \iid R$. For all $t \in [n-1]$, let $p_t^\texttt{CONF}$ be as in line 18 of Algorithm~\ref{alg: CONCH} and $C_{n,1-\alpha}^\texttt{MCP}$ be as in (\ref{eq:CONCH CI}). Let $l_{n,1-\alpha}$ be the size of $C_{n,1-\alpha}^\texttt{MCP}$ i.e. $l_{n,1-\alpha} =  \sum_{t=1}^{n-1} \mathbf{1}(t \in C_{n,1-\alpha}^\texttt{MCP})$.  
    Then for all $t \in [n-1]$, $p_t^\texttt{CONF}$ has the CDF $F$ defined as 
    \[
F(y) = 
\begin{cases}
0, & y < 0, \\[6pt]
1 - \left(1 - \tfrac{y}{2}\right)^2, & 0 \leq y < 1, \\[6pt]
1, & y \geq 1 ,
\end{cases}
\]
and $\E(l_{n,1-\alpha}) = (1 - \frac{\alpha}{2})^2(n-1).$

\end{proposition}
It is noteworthy that, in case $\tau_n = n$, $p_t^\texttt{CONF}$ has an atom at $1$ since $\P(p_t^\texttt{CONF} = 1) = \frac{1}{4}$, for all $t \in [n-1]$, and a very large fraction of the datapoints stay in the confidence set for $\tau_n$ which is illustrated in \Cref{fig:exchn}. 

In practice, to construct $W_t^{(0)}$ and $W_t^{(1)}$ (line 15 of Algorithm~\ref{alg: CONCH}), one can also use other distribution-free tests instead of the Kolmogorov-Smirnov test, such as the Anderson-Darling test or Cramér–von Mises test. Furthermore, 
by Donsker's theorem (\cite{donsker1951invariance}), as $t,n-t \to \infty$, $W_t^{(0)} \convD \sup_{z \in [0,1]}|\phi_z|, \quad W_t^{(1)} \convD \sup_{z \in [0,1]}|\phi_z|,$
where $\phi_z = B_z - zB_1,$ with $(B_z)_{z \geq 0}$ being the standard Brownian motion on $\R$. Thus, one can use this for approximating the CDF of $W_t^{(0)}$ and $W_t^{(t)}$ to perform the probability integral transform (in line 16 of Algorithm~\ref{alg: CONCH}), for large values of $t$ and $n-t$.


\section{Simulation Studies}
\label{simulation studies}
We now present a series of simulation studies designed to corroborate the performance of the MCP algorithm and to illustrate the theoretical properties established in the preceding sections. These experiments are conducted across a variety of settings, including both mean and non-mean changes, and under diverse distributional scenarios. We demonstrate that the algorithm maintains strong precision for the changepoint estimate and the scaled size of the confidence set for the changepoint converges to $0$ as the sample size goes to $\infty$ without relying on restrictive assumptions. Since the theoretical results are obtained using the identity score function, the same has been used in the MCP algorithm to obtain the plots of this section. Also, throughout, the miscoverage level $\alpha=0.05$.

\Cref{fig:nprmal pval} illustrates the characteristic sharp spike of MCP $p$-values around the true changepoint $\tau_n$ for a variance shift ($N(0,1)$ to $N(0,5)$). Besides the distribution-change setup, to illustrate the behavior of $p$-values of the MCP Algorithm under exchangeability, we refer to \Cref{fig:exchn}. The patterns seen in the diagram can be explained from Proposition~\ref{prop:pval exchan}. Since under exchangeability, the distribution of $p_t^\texttt{CONF}$ has an atom at $1$, in \Cref{fig:exchn}, for many time points $t$, $p_t^\texttt{CONF}$ is $1$. In addition, since the expected size of the confidence set produced by the MCP Algorithm is $(1 - \frac{0.05}{2})^2(200 - 1) \approx 189$, almost all  points are included in the confidence set. Now, for different pre- and post-change distributions, we find the relative size of the confidence set produced by the MCP Algorithm using identity score function for samples of sizes growing from $2$ upto $4000$ in \Cref{fig:rel_length_CONCH}. To illustrate various distribution changes, we examine cases where pre- and post-change distributions share the same family but differ in variance ($N(0,1)$ vs. $N(0,5)$), mean (Cauchy$(0,1)$ vs. Cauchy$(5,1)$), or both (Exp$(1)$ vs. Exp$(5)$), as well as cases involving different parametric families ($N(0,1)$ vs. Cauchy$(5,1)$). 

From \Cref{fig:rel_length_CONCH}, it can be seen that for all the cases, the relative length of the confidence set produced by the MCP Algorithm tends to $0$ as the sample size $n$ tends towards $\infty$, which is in accordance with \Cref{th: rel len CONCH}. \Cref{fig:consistency} validates the consistency of $\widehat{\tau}_n$ by displaying the ratio $\widehat{\tau}_n/\tau_n$ for sample sizes up to $8000$ with $\tau_n = 0.5n$. Across all distributional shifts, the ratio converges to $1$ as $n \to \infty$, empirically corroborating the asymptotic guarantee of \Cref{th: rel len CONCH}. Finally, we provide extensive additional validations across diverse distribution shifts and score functions in \Cref{fig:norm_norm} - \ref{fig:norm_cauchy}. These studies illustrate that the MCP algorithm maintains reliable performance across a wide range of data-generating processes under the distributional change setup.
\begin{figure}
    \centering
    \includegraphics[width=\linewidth]{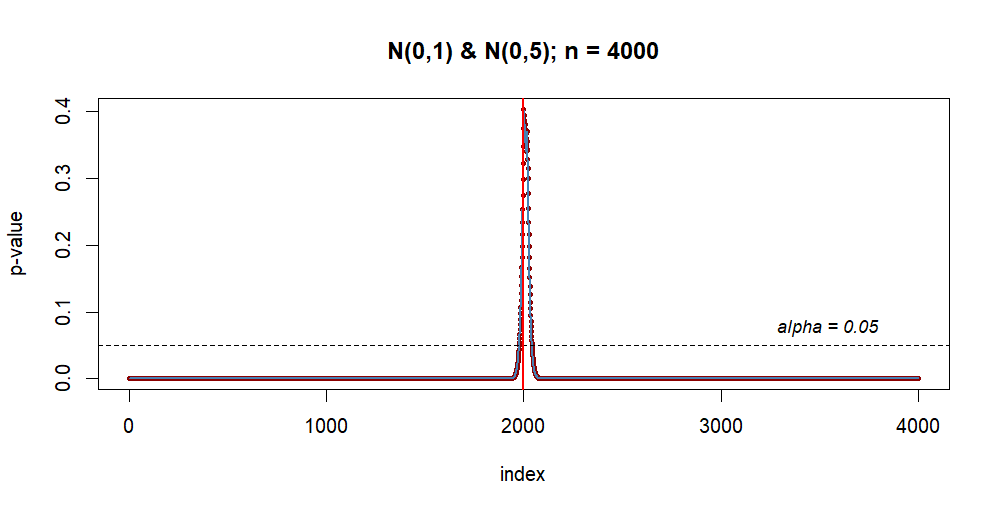}
    \caption{$p$-values produced by the \texttt{MCP Algorithm}  (with the score function $s(z) = z$) for $N(0,1)$ pre-change and $N(0,5)$ post-change distribution on a sample of size $n = 4000$ and true changepoint at $\tau_n = 0.5n$ (marked by red line). In the presence of a distribution change, the $p$-values exhibit a sharp spike-type pattern around the true changepoint.}
    \label{fig:nprmal pval}
\end{figure}
\begin{figure}
    \centering
    \includegraphics[width=0.45\linewidth]{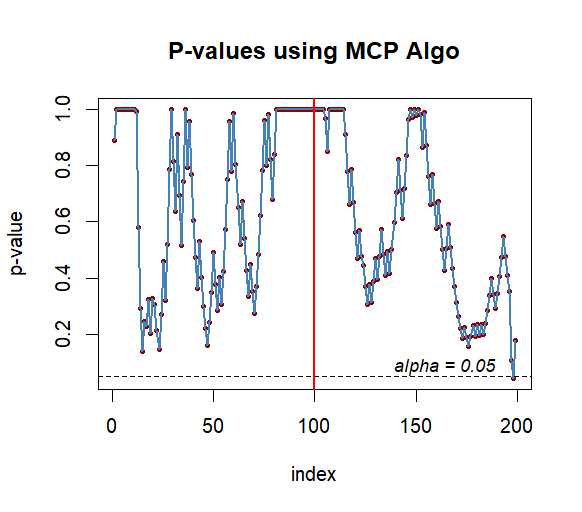}
    \includegraphics[width=0.45\linewidth]{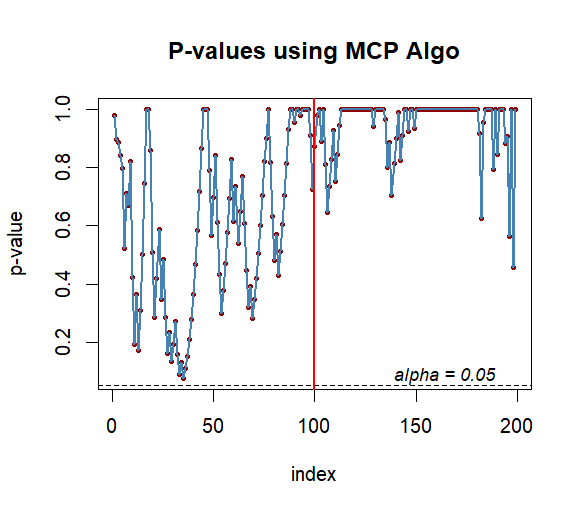}
    \caption{$p$-values produced by the \texttt{MCP Algorithm} (with the score function $s(z) = z$) for a sample of size $200$ exchangeably drawn from $N(0,1)$ (left) and Exp$(1)$ (right). This figure demonstrates Proposition \ref{prop:pval exchan} and illustrates the existence of atoms in the marginal distribution of the $p$-values under exchangeability of the whole sample.}
    \label{fig:exchn}
\end{figure}
\begin{figure}
    \centering
    \includegraphics[width=1.3\linewidth, angle=-90]{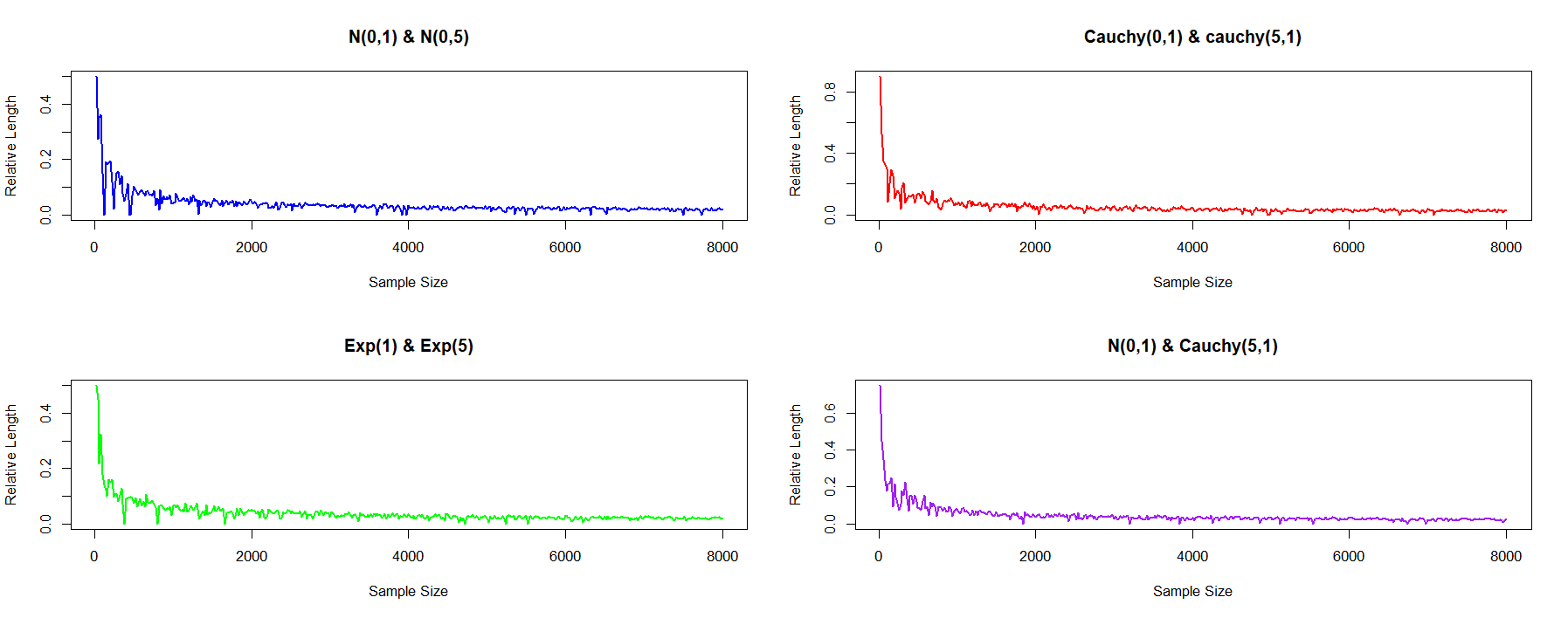}
    \caption{Relative length of the CI of \texttt{MCP Algorithm} for mentioned pre and post-change distributions, sample size $n$, change-point at $\tau_n = \frac{n}{2}$ i.e. $c = 0.5$, significance level $\alpha = 0.05$ and score function $s(z) = z$. For all the cases, the relative length is seen to converge to $0$, as guaranteed in \Cref{th: rel len CONCH}.}
    \label{fig:rel_length_CONCH}
\end{figure}
\begin{figure}
    \centering
    \includegraphics[width=0.45\linewidth]{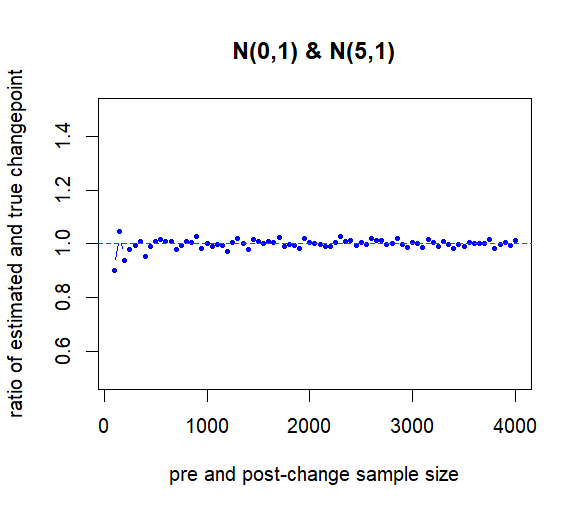}
    \includegraphics[width=0.45\linewidth]{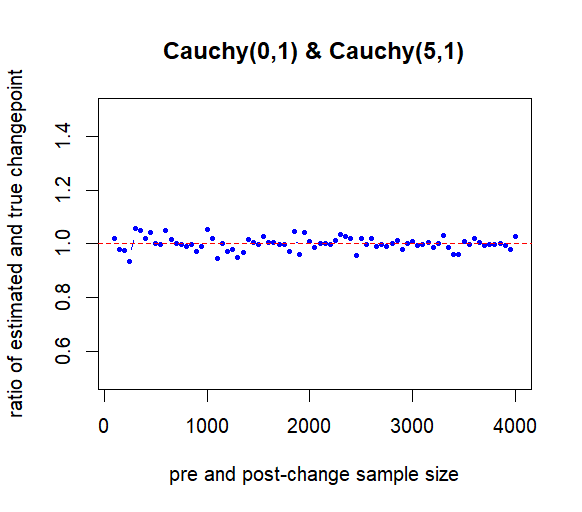}
    \includegraphics[width=0.45\linewidth]{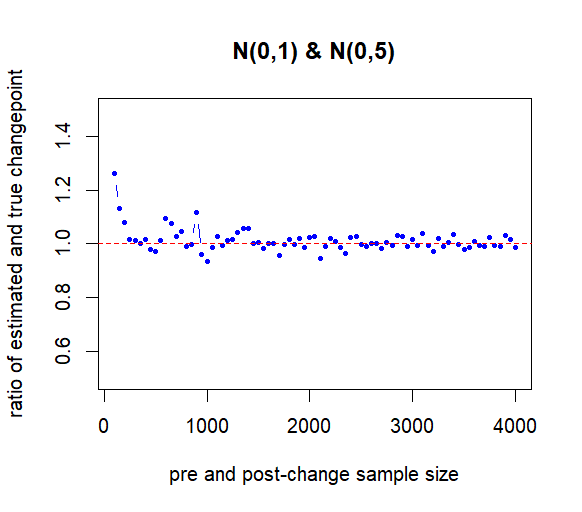}
    \includegraphics[width=0.45\linewidth]{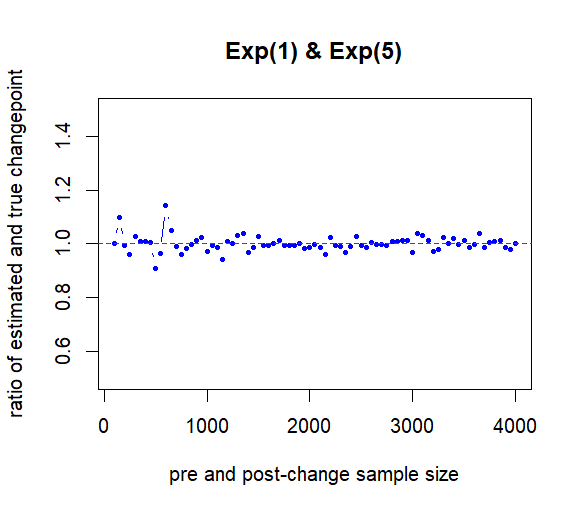}
    \includegraphics[width=0.45\linewidth]{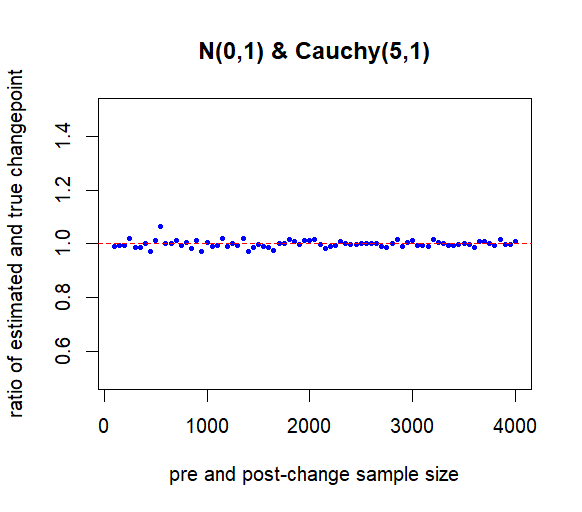}
    \includegraphics[width=0.45\linewidth]{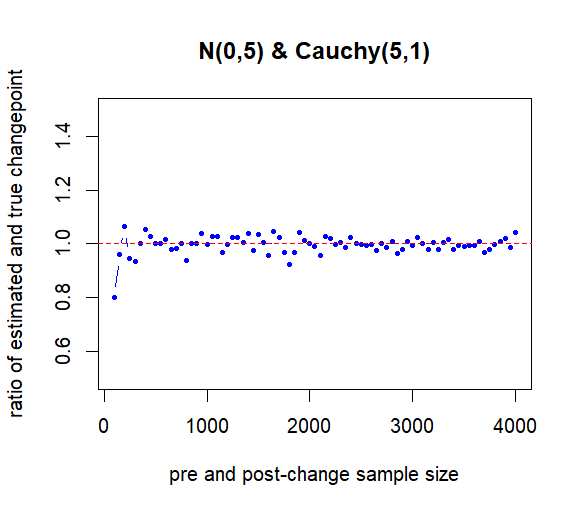}
    \caption{Ratio of estimated change-point ($\widehat{\tau}_n$) and true change-point ($\tau_n = cn, c = 0.5$) for mentioned pre and post-change distributions with $n$ to be the total sample size. As the sample size increases, the ratio $\frac{\widehat{\tau}_n}{\tau_n}$ is seen to converge to $1$, as guaranteed in \Cref{th: rel len CONCH}.}
    \label{fig:consistency}
\end{figure}
\begin{figure}
    \centering
    \includegraphics[width=\linewidth]{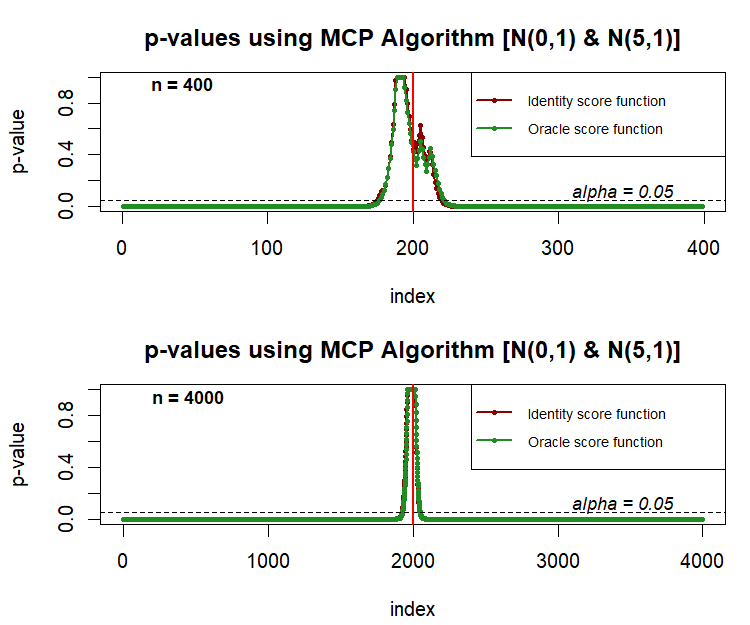}
    \caption{$p$-values of the \texttt{MCP Algorithm} for sample of size $n = 400$(above), $n = 4000$(below), changepoint at $\tau_n = \frac{n}{2}$ i.e. $c = 0.5$, pre and post-change distributions $\mathcal{N}(0,1)$ and $\mathcal{N}(5,1)$ respectively, for identity and oracle score functions. The oracle score function yields a narrower confidence set for $\tau_n$ as compared to the identity score function. As the sample size increases, the confidence set size decreases. }
    \label{fig:norm_norm}
\end{figure}
\begin{figure}
    \centering
    \includegraphics[width=\linewidth]{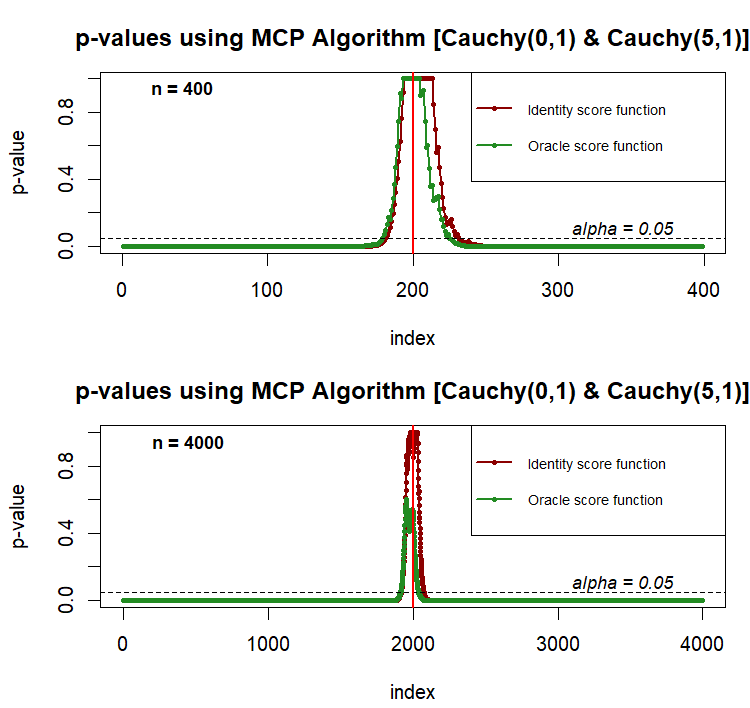}
    \caption{$p$-values of the \texttt{MCP Algorithm} for sample of size $n = 400$(above), $n = 4000$(below), changepoint at $\tau_n = \frac{n}{2}$ i.e. $c = 0.5$, pre and post-change distributions Cauchy$(0,1)$ and Cauchy$(5,1)$ respectively, for identity and oracle score functions. The oracle score function yields a narrower confidence set for $\tau_n$ as compared to the identity score function. As the sample size increases, the confidence set size decreases. }
    \label{fig:cauchy_cauchy}
\end{figure}
\begin{figure}
    \centering
    \includegraphics[width=\linewidth]{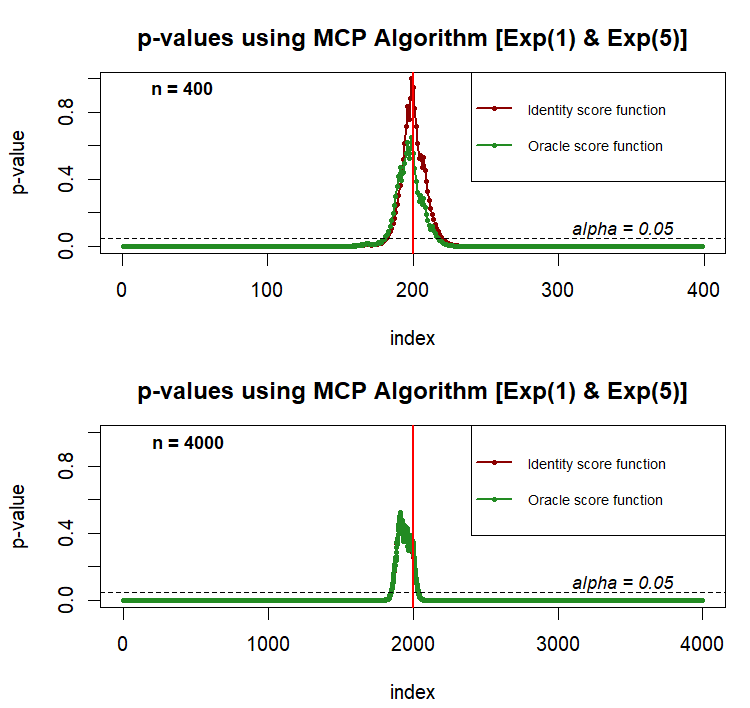}
    \caption{$p$-values of the \texttt{MCP Algorithm} for sample of size $n = 400$(above), $n = 4000$(below), changepoint at $\tau_n = \frac{n}{2}$ i.e. $c = 0.5$, pre and post-change distributions Exp$(1)$ and Exp$(5)$ respectively, for identity and oracle score functions. The oracle score function yields a narrower confidence set for $\tau_n$ as compared to the identity score function. As the sample size increases, the confidence set size decreases. }
    \label{fig:exp_exp}
\end{figure}
\begin{figure}
    \centering
    \includegraphics[width=\linewidth]{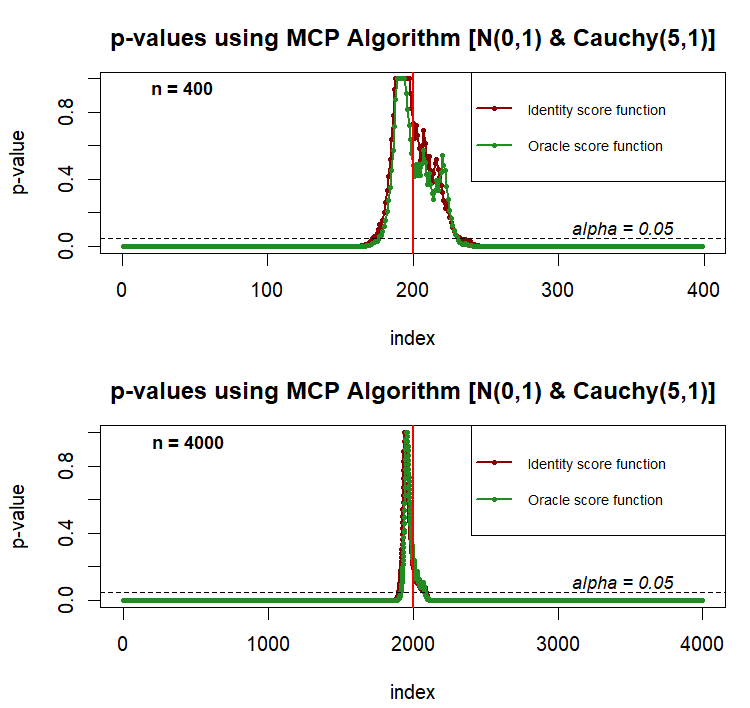}
    \caption{$p$-values of the \texttt{MCP Algorithm} for sample of size $n = 400$(above), $n = 4000$(below), changepoint at $\tau_n = \frac{n}{2}$ i.e. $c = 0.5$, pre- and post-change distributions $\mathcal{N}(0,1)$ and Cauchy$(5,1)$ respectively, for identity and oracle score functions. The oracle score function yields a narrower confidence set for $\tau_n$ as compared to the identity score function. As the sample size increases, the confidence set size decreases. }
    \label{fig:norm_cauchy}
\end{figure}

\clearpage
\newpage

\section{Optimal Score Functions for different kinds of Distribution Changes}
\label{Appendix A}
The performance of the conformal $p$-value-based method depends on the choice of the underlying score function. As shown in the following result, in the distribution-change setup, the likelihood ratio of the post and pre-change distribution happens to be the most optimal score function yielding the highest power.

\begin{fact}[Conformal Neyman Pearson Lemma (\cite{dandapanthula2025conformal})]\\ 
\label{Th:Th1}
    Suppose $X_1,\cdots,X_\tau \iid R$ and $X_{\tau+1},\cdots,X_n \iid Q$ where $Q,R$ are absolutely continuous with respect to the Lebesgue measure and have densities $q,r$ respectively. For any real-valued Borel measurable function $s$, define 
    \begin{equation}
        \label{conf pval}
        p_n[s] = \frac{1}{n} \sum_{i=1}^n \bigg( \mathbf{1}\big(s(X_i) > s(X_n) \big) + \theta_n \mathbf{1}\big( s(X_i) = s(X_n) \big) \bigg),
    \end{equation}
    where $\theta_n \sim U(0,1)$ and is independent of $(X_1,\cdots,X_n)$. Then, for the oracle score function $s^*(x) = \frac{q(x)}{r(x)}$, for Lebesgue-almost every $u \in (0,1)$, we have 
    \[ \E\bigg[p_n[s^*] \bigg| s^*(X_n) = F_{s^*(X_n)}^{-1}(u) \bigg] \geqslant \E\bigg[p_n[s] \bigg|s(X_n) = F_{s(X_n)}^{-1}(u)\bigg], \] where $F_{s(X_n)}^{-1}$ denotes the quantile function of $s(X_n)$.  
\end{fact}

\Cref{Th:Th1} has an useful implication for tests based on conformal $p$-values. Suppose, we observe independent random variables $X_1,\cdots,X_n$ and want to test $\mathcal{H}_0: X_1,\cdots,X_n \hspace{0.3cm} \text{are i.i.d.},$
versus the alternative $\mathcal{H}_1: X_1,\cdots,X_\tau \iid R, \hspace{0.2cm} X_{\tau+1},\cdots,X_n \iid Q,$
using $p_n[s]$. In case the likelihood ratio $\frac{q(x)}{r(x)}$ yields the most powerful conformal test. Now, when the distributional change occurs in a very particular way, it is customary to find the optimal score function. The most common changes in this regard can be changes in the covariate, Label, or the regression function when we have paired observations. In the subsequent sections, we explore the most optimal score functions in the above changes and discuss the general properties in that context. 
\subsection{Covariate Shift}
When we have paired data $\{(X_i,Y_i)\}_{i=1}^n$, covariate shift refers to the setup where the distribution of $X$ changes after the changepoint occurs but the distribution of $Y|X$ remains throughout the data stream. So formally, in our case if $\tau$ denotes the changepoint, then $(X_1,Y_1),\cdots,(X_\tau,Y_\tau) \iid R_X \times R_{Y|X}$ and $(X_{\tau+1},Y_{\tau+1}),\cdots,(X_n,Y_n) \iid Q_X \times R_{Y|X}$. Here, the notation $R_X \times R_{Y|X}$ means that $X$ follows distribution $R_X$ and then, conditionally on $X$, $Y \mid X$ follows distribution $R_{Y|X}$ (and, $Q_X \times R_{Y|X}$ is defined similarly).

 The key insight in this setup is if $R$ and $Q$ denotes the pre-chage and post-change distributions of the paired observations respectively, then the likelihood ratio $\frac{dQ}{dR}$ is same with $\frac{dQ_X}{dR_X}$, so it does not depend on $R_{Y|X}$. Formally, $\frac{dQ_X}{dR_X}$ denotes the Radon–Nikodym derivative relating the distribution $Q_X$ of the pre-change data point’s feature vector to the distribution $R_X$ of the post-change features. For example, if $R_X$ and $Q_X$ both have densities $r_X$ and $q_X$ respectively with respect to a common measure (e.g. Lebesgue Measure), then $\frac{dQ_X}{dR_X}$ can simply be taken to be the ratio of the densities $\frac{q_X}{r_X}$. This motivates us to find the optimal score function in Covariate Shift, which we formally state in the following theorem. 

 \begin{apptheorem}[Conformal Neyman Pearson Lemma for Covariate Shift]\\
 \label{th:NP Cov}
Suppose $(X_1, Y_1), \ldots, (X_\tau, Y_\tau) \overset{\text{iid}}{\sim} R_X \times R_{Y|X}$ and $(X_{\tau+1}, Y_{\tau+1}), \ldots, (X_n, Y_n) \overset{\text{iid}}{\sim} Q_X \times R_{Y|X}$, where $Q_X \times R_{Y|X}, R_X \times R_{Y|X}$ are absolutely continuous with respect to the Lebesgue measure and have densities $q, r$ respectively. Let $q_X, r_X$ be marginal densities of $X$ for $q, r$ respectively. For any real-valued Borel measurable function $s$, let $p_n[s]$ be as in (\ref{conf pval}).
Define the \textit{$X-$marginal likelihood conformity score} as,
\[
s_X^*(x,y) = \frac{q_X(x)}{r_X(x)}.
\]
Then
\[
\E\left[ p_n[s_X^*] \right] \geq \E\left[ p_n[s] \right].
\]

 \end{apptheorem}
 \begin{pf}
     The proof of Theorem~\ref{th:NP Cov} follows from \Cref{Th:Th1} and the observation that for the assumptions, 
 \[ \frac{d(Q_X \times R_{Y|X})}{d(R_X \times R_{Y|X})}(x,y) = \frac{q_X(x)}{r_X(x)}. \]
 Thus, in this setup, the optimal score function turns out to be the ratio of densities of the covariate. \qed

 \end{pf}

\subsection{Label Shift}
When we have paired data $\{(X_i,Y_i)\}_{i=1}^n$, label shift refers to the setup where the distribution of $Y$ changes and the distribution of $X \mid Y$ is held fixed throughout the data stream. Label shift is common in many classification problems. Formally, in our case if $\tau$ denotes the changepoint, then $(X_1,Y_1),\cdots,(X_\tau,Y_\tau) \iid R_{X|Y} \times R_Y$ and $(X_{\tau+1},Y_{\tau+1}),\cdots,(X_n,Y_n) \iid R_{X|Y} \times Q_Y$. Here, the notation $R_{X|Y} \times R_Y$ means that $Y$ follows distribution $R_Y$ and then $X \mid Y$ follows distribution $R_{X|Y}$, and similarly, for $R_{X|Y} \times Q_Y$.

 The treatment of this case is similar to Covariate Shift.
 If $R$ and $Q$ denotes the pre-chage and post-change distributions of the paired observations respectively, then the likelihood ratio $\frac{dQ}{dR}$ is same with $\frac{dQ_Y}{dR_Y}$, so it does not depend on $R_{X|Y}$. The next theorem, which is similar to Theorem~\ref{th:NP Cov}, gives the optimal score function for the Label Shift setup. 

 \begin{apptheorem}[Conformal Neyman Pearson Lemma for Label Shift]\\
 \label{th:NP Label}
     Suppose $(X_1, Y_1), \ldots, (X_\tau, Y_\tau) \overset{\text{iid}}{\sim} R_{X|Y} \times R_Y$ and $(X_{\tau+1}, Y_{\tau+1}), \ldots, (X_n, Y_n) \overset{\text{iid}}{\sim} R_{X|Y} \times Q_Y$, where $R_{X|Y} \times Q_Y, R_{X|Y} \times R_Y$ are absolutely continuous with respect to the Lebesgue measure and have densities $q, r$ respectively. Let $q_Y, r_Y$ be marginal densities of $Y$ for $q, r$ respectively. For any real-valued Borel measurable function $s$, let $p_n[s]$ be as in (\ref{conf pval}). Define the $Y-$\textit{marginal likelihood conformity score} as,
     \[ s^*_Y(x,y) = \frac{q_Y(y)}{r_Y(y)}. \]
      Then
\[
\E\left[ p_n[s_Y^*] \right] \geq \E\left[ p_n[s] \right].
\]
 \end{apptheorem}
 \begin{pf}
     The proof of Theorem~\ref{th:NP Label} follows from \Cref{Th:Th1} and the observation that for the assumptions, 
 \[ \frac{d(R_{X|Y} \times Q_Y)}{d(R_{X|Y} \times R_Y)}(x,y) = \frac{q_Y(y)}{r_Y(y)}. \]
 Thus, in this setup, the optimal score function turns out to be the ratio of densities of the Label. \qed

 \end{pf}

\subsection{Regression Function Shift}
When we have paired data $\{(X_i,Y_i)\}_{i=1}^n$, regression function shift refers to the setup where the distribution of $Y|X$ changes and the distribution of $X$ is held fixed throughout the data stream. The Regression Function shift is common in many regression discontinuity problems. Formally, in our case if $\tau$ denotes the changepoint, then $(X_1,Y_1),\cdots,(X_\tau,Y_\tau) \iid R_{X} \times R_{Y|X}$ and $(X_{\tau+1},Y_{\tau+1}),\cdots,(X_n,Y_n) \iid R_{X} \times Q_{Y|X}$. Here, the notation $R_{X} \times R_{Y|X}$ means that $X$ follows distribution $R_X$ and contional on $X$, $Y|X$ follows distribution $R_{Y|X}$, and similarly, for $R_{X} \times Q_{Y|X}$.

 If $R$ and $Q$ denote the pre-change and post-change distributions of the paired observations respectively, then the likelihood ratio $\frac{dQ}{dR}$ is the same as $\frac{dQ_{Y|X}}{dR_{Y|X}}$, so it does not depend on $R_{X}$. The next theorem gives the optimal score function for the Regression Function Shift setup. 

 \begin{apptheorem}[Conformal Neyman Pearson Lemma for Regression Function Shift]\\
 \label{th:NP Reg}
     Suppose $(X_1,Y_1),\cdots,(X_\tau,Y_\tau) \iid R_{X} \times R_{Y|X}$ and $(X_{\tau+1},Y_{\tau+1}),\cdots,(X_n,Y_n) \iid R_{X} \times Q_{Y|X}$, where $R_{X} \times Q_{Y|X}, R_{X} \times R_{Y|X}$ are absolutely continuous with respect to the Lebesgue measure and have densities $q, r$ respectively. Let $q_{Y|X}, r_{Y|X}$ be conditional densities of $Y|X$ for $q, r$ respectively. For any real-valued Borel measurable function $s$, let $p_n[s]$ be as in (\ref{conf pval}). Define the \textit{$Y|X-$conditional likelihood conformity score} as, 
\[
s_{Y|X}^*(x,y) = \frac{q_{Y|X}(y|x)}{r_{Y|X}(y|x)}. 
\]
 Then
\[
\E\left[ p_n[s_{Y|X}^*] \right] \geq \E\left[ p_n[s] \right].
\]
 \end{apptheorem}

  \begin{pf}
     The proof of Theorem~\ref{th:NP Reg} follows from \Cref{Th:Th1} and the observation that for the assumptions, 
 \[ \frac{d(R_X \times Q_{Y|X})}{d(R_X \times R_{Y|X})}(x,y) = \frac{q_{Y|X}(y|x)}{r_{Y|X}(y|x)}.\]
 Thus, in this setup, the optimal score function turns out to be the ratio of conditional densities of the label given the covariate. \qed

 \end{pf}

 \subsection{Optimal Score Function for Multiple Change-points}
We now discuss the case when there are multiple change-points. In this setup, we have observations $X_1,\cdots,X_n$ with change-points $\tau_1,\tau_2,\cdots,\tau_k$ and distribution functions $R_1,\cdots,R_k,Q$ with $\{X_i : 1 \leqslant i \leqslant \tau_1\}$ marginally follow distribution $R_1$, $\{X_i : \tau_1+1 \leqslant i \leqslant \tau_2\}$ marginally follow distribution $R_2$, $\cdots ,\{X_i : \tau_{k-1}+1 \leqslant i \leqslant \tau_k\}$ marginally follow distribution $R_k$ and $\{X_i : \tau_k+1 \leqslant i \leqslant n\}$ marginally follow distribution $Q$. One can assume the joint distribution of the observations in addition (e.g. $X_1,\cdots,X_n$ are independent).In the next theorem, we find the optimal score function to be used in the conformal $p$-value for this setup, which generalizes \Cref{Th:Th1} for multiple change-points.  First, we prove a particular form of the Neyman-Pearson Lemma, which will be used in the next theorem. The following version of the Neyman-Pearson Lemma attempts to provide the level $\alpha$ test, $\alpha \in (0,1),$ which maximizes the weighted power for a simple null and composite alternative.

\begin{applemma}[Neyman Pearson Lemma for Weighted power maximization against composite alternative]
\label{Th: NP weighted power}
Let $R_1,\cdots,R_k,Q$ be probability distributions on $\R^k$ having densities $r_1,\cdots,r_k,q$ respectively with respect to the Lebesgue Measure. Let $\alpha \in (0,1)$. Consider the hypothesis:
\begin{equation}
    \begin{split}
        H_0&: X \sim Q, \\
        H_1&: X \sim \bigcup_{j=1}^k R_j.
    \end{split}
\end{equation}
Let $w_j \geq 0$ be nonnegative weights. Consider the test function $\varphi : \R^k \to [0,1]$ given by 
\begin{equation}
    \varphi(x) = 
    \begin{cases} 
        1 & \text{if } \sum_{j=1}^k w_j r_j(x) > c \cdot q(x), \\
        \gamma & \text{if } \sum_{j=1}^k w_j r_j(x) = c \cdot q(x), \\
        0 & \text{if } \sum_{j=1}^k w_j r_j(x) < c \cdot q(x),
    \end{cases}
    \label{eq:test}
\end{equation}
where $c \geq 0$ and $\gamma \in [0,1]$ satisfy $\mathbb{E}_{X \sim Q}[\varphi(X)] = \alpha$. Then among all test functions $\varphi' : \R^k \to [0,1]$ satisfying $\mathbb{E}_{X \sim Q}[\varphi'(X)] \leq \alpha$, the test function $\varphi$ maximizes the \texttt{weighted average power} $\sum_{j=1}^k w_j \, \mathbb{E}_{X \sim R_j}[\varphi'(X)]$ i.e. 
\begin{equation}
    \sum_{j=1}^k w_j \, \mathbb{E}_{X \sim R_j}[\varphi(X)] \geq \sum_{j=1}^k w_j \, \mathbb{E}_{X \sim R_j}[\varphi'(X)].
    \label{eq:objective}
\end{equation}
\end{applemma}
\begin{proof}
If \(W:=\sum_{j=1}^k w_j=0\) the conclusion is trivial (both sides are zero).
Assume \(W>0\) and define the mixture density
\[
r^*(x)\;:=\;\frac{1}{W}\sum_{j=1}^k w_j r_j(x).
\]
Observe that for any measurable \(\psi:\mathbb{R}^d\to[0,1]\),
\[
\sum_{j=1}^k w_j \,\mathbb{E}_{R_j}[\psi(X)]
\;=\;W\,\mathbb{E}_{R^*}[\psi(X)].
\]
Hence, maximizing the weighted average power over all tests \(\varphi'\)
with \(\mathbb{E}_Q[\varphi']\le\alpha\) is equivalent to maximizing
\(\mathbb{E}_{R^*}[\varphi']\) under the same level constraint. By the Neyman--Pearson lemma (\cite{lehmann2005testing}),
the likelihood-ratio test
\[
\varphi_{LR}(x)=
\begin{cases}
1, & \text{if } \dfrac{r^*(x)}{q(x)} > t,\\[6pt]
\gamma, & \text{if } \dfrac{r^*(x)}{q(x)} = t,\\[6pt]
0, & \text{if } \dfrac{r^*(x)}{q(x)} < t,
\end{cases}
\]
with constants \(t>0\) and \(\gamma\in[0,1]\) chosen so that \(\mathbb{E}_Q[\varphi_{LR}]=\alpha\),
maximizes \(\mathbb{E}_{R^*}[\varphi]\) among all level-\(\alpha\) tests, which is equivalent to the test $\varphi$ defined in (\ref{eq:test}) by taking $c = Wt$. Hence, we have,
\[
\sum_{j=1}^k w_j \,\mathbb{E}_{R_j}[\varphi(X)]
= W\,\mathbb{E}_{R^*}[\varphi(X)]
\ge W\,\mathbb{E}_{R^*}[\varphi'(X)]
= \sum_{j=1}^k w_j \,\mathbb{E}_{R_j}[\varphi'(X)]
\]
for every measurable \(\varphi': \R^k \to [0,1]\) with \(\mathbb{E}_Q[\varphi']\le\alpha\).
This completes the proof.
\end{proof}

 \begin{apptheorem}[Generalized Conformal Neyman Pearson Lemma for Multiple Change-points]
 \label{th: CONF NP Lemma mujlti chnagepoint}
     Suppose $R_1,\cdots,R_k,Q$ are probability distributions on $\R^k$ each of which is absolutely continuous with respect to the Lebesgue Measure on $\R^k$ and has densities $r_1,\cdots,r_k,q$ respectively. $X_1,\cdots,X_n$ are random variables satisfying $X_1,\cdots,X_{\tau_1} \iid R_1, \quad X_{\tau_1+1},\cdots,X_{\tau_2} \iid R_2, \cdots,$ $
     X_{\tau_{k-1}+1},\cdots,X_{\tau_k} \iid R_k, \quad X_{\tau_k+1},\cdots,X_n \iid Q$ where $1 < \tau_1 < \cdots < \tau_k < n$. For any Borel measurable function $s : \R^k \to \R$ define, 
     \[ p_n[s] = \frac{1}{n}\sum_{i=1}^n \mathbf{1}\big( s(X_i) > s(X_n) \big). \]
     Suppose in addition $s$ is such that $s(X_{\tau_k + 1}),s(X_n)$ are distinct almost surely. Moreover, consider the \textit{oracle conformity score}, 
     \[ s^*(x) = \frac{q(x)}{ \frac{\tau_1}{\tau_k}r_1(x) + \frac{\tau_2 - \tau_1}{\tau_k}r_2(x) + \cdots + \frac{\tau_k - \tau_{k-1}}{\tau_k}r_k(x)   }, \]
     and suppose that $s^*(X_{\tau_k + 1}),s^*(X_n)$ are distinct almost surely. Then we have 
     \[ \E[p_n[s^*]] \geqslant \E[p_n[s]].\]
 \end{apptheorem}

 \begin{pf}
     Since $X_1,\cdots,X_n$ are random variables satisfying $X_1,\cdots,X_{\tau_1} \iid R_1, \quad X_{\tau_1+1},\cdots,X_{\tau_2} \iid R_2, \cdots, \quad 
     X_{\tau_{k-1}+1},\cdots,X_{\tau_k} \iid R_k, \quad X_{\tau_k+1},\cdots,X_n \iid Q$, 
     Now for
\[
p_n[s] = \frac{1}{n} \sum_{i=1}^n \mathbf{1}\{s(X_i) > s(X_n)\},
\]
observe that since $s(X_{\tau_k+1}),s(X_n)$ are assumed to be distinct almost surely, hence 
\[ \P\bigg( s(X_{\tau_k+1}) > s(X_n) \bigg) = \frac{1}{2}. \]
So, with $U \sim U(0,1)$ and $U$ independent of $X_1,\cdots,X_n$, we can write
\begin{align*}
\E[p_n[s]] &= \frac{\tau_1}{n} \P\bigg(s(X_{\tau_1}) > s(X_n)\bigg) 
+ \frac{\tau_2 - \tau_1}{n} \P\bigg(s(X_{\tau_2}) > s(X_n)\bigg) 
+ \cdots 
+ \frac{n - \tau_K}{n} \cdot \frac{1}{2} \\
&= \frac{\tau_1}{n} \P\bigg(F_{s(X_n)}(s(X_{\tau_1})) > U\bigg) 
+ \frac{\tau_2 - \tau_1}{n} \P\bigg(F_{s(X_n)}(s(X_{\tau_2})) > U\bigg) 
+ \cdots \\
&\quad + \frac{\tau_k - \tau_{k-1}}{n} \P\bigg(F_{s(X_n)}(s(X_{\tau_k})) > U\bigg) 
+ \frac{n - \tau_K}{2n} \\
& \text{[By the Probability Integral Transform; } F_{s(X_n)}(s(X_n)) \sim U \text{ with } U \sim U(0,1) \text{]} \\
&= \frac{\tau_1}{n} \P\bigg(s(X_{\tau_1}) > F_{s(X_n)}^{-1}(U)\bigg) 
+ \frac{\tau_2 - \tau_1}{n} \P\bigg(s(X_{\tau_2}) > F_{s(X_n)}^{-1}(U)\bigg) 
+ \cdots \\
&\quad + \frac{\tau_k - \tau_{k-1}}{n} \P\bigg(s(X_{\tau_k}) > F_{s(X_n)}^{-1}(U)\bigg) 
+ \frac{n - \tau_k}{2n} \\
&= \sum_{i=1}^k \omega_i \, \E_{Z \sim R_i} \left[ \mathbf{1}(s(Z) > F_{s(X_n)}^{-1}(U)) \right] 
+ \frac{n - \tau_k}{2n} \\
&\quad \text{where } \omega_1 = \frac{\tau_1}{n}, \quad 
\omega_i = \frac{\tau_i - \tau_{i-1}}{n}, \quad 2 \leq i \leq k \\
&= \int_0^1 \left[ \sum_{i=1}^k \omega_i \, \E_{Z \sim R_i} \left[ \mathbf{1}(s(Z) > F_{s(X_n)}^{-1}(1-u)) \right] \right] \, du 
+ \frac{n - \tau_k}{2n}.
\end{align*}
Fix \( u \in (0,1) \) and consider the test
\[
\phi = \mathbf{1}\bigg(s(Z) > F_{s(X_n)}^{-1}(1-u)\bigg).
\]
for testing \( \tilde{H}_0: Z \sim Q \) against \( \tilde{H}_1: Z \sim \bigcup_{j=1}^k R_j \).  
By the probability integral transform lemma, the size of this test is
\[
\P_{Z \sim Q}(s(Z) > F_{s(X_n)}^{-1}(1-u)) = u
\]
Hence by the \textit{Neyman–Pearson Lemma for weighted power maximization against composite alternative}(Lemma~\ref{Th: NP weighted power}), the power of this test is maximized for \( s = \tilde{s}^* \), where \( \tilde{s}^* \) is
\[
\tilde{s}^*(X) = \frac{q(X)}{\sum_{i=1}^k \omega_i r_i(X)}.
\]
So equivalently,
\[
s^*(X) = \frac{q(X)}{\frac{\tau_1}{\tau_k} r_1(X) + \cdots + \frac{\tau_k - \tau_{k-1}}{\tau_k} r_k(X)},
\]
maximizes \( \E[p_n[s]] \) since $p_n[\Tilde{s}^*(X)] = p_n[s^*(X)]$. \qed
 \end{pf}

\newpage
\section{MCP Algorithm For a Triangular Array of Random Variables}
\label{sec: Appendix B}
Given a triangular array of data $\{X_{n,j}\}_{n \in \N, 1 \leqslant j \leqslant n}$, the version of the changepoint detection algorithm mentioned in Algorithm~\ref{alg: CONCH} and changepoint localization algorithm mentioned in \cite{dandapanthula2025conformal} to apply on $\{X_{n,j}\}_{1 \leqslant j \leqslant n}$ is the follwing:

\medskip
\begin{algorithm2e}[H]
\caption{MCP Algorithm for Changepoint Detection and Localization}
\label{alg: CONCH traiangular}
\KwIn{$(X_{n,t})_{t=1}^n$ (dataset), 
      $\big((s^{(0)}_{n,r,t})_{r=1}^{n-1}\big)_{t=1}^{n-1}$ and 
      $\big((s^{(1)}_{n,r,t})_{r=1}^{n-1}\big)_{t=1}^{n-1}$ 
      (left and right score function families)}
\KwOut{$\widehat{\tau}_n$ (estimate of changepoint $\tau_n$) and $C^{\texttt{MCP}}_{n,1-\alpha}$ (a level $1-\alpha$ confidence set for $\tau_n$)}

\For{$t \in [n-1]$}{
    \For{$r \in (1, \ldots, t)$}{
        \For{$j \in (1, \ldots, r)$}{
            $\kappa^{(t)}_{n,r,j} \leftarrow s^{(0)}_{n,r,t}\!\left(X_{n,j}\,;\,[X_{n,1}, \ldots, X_{n,r}],\,(X_{n,t+1}, \ldots, X_{n,n})\right)$\;
        }
        \textbf{end}\;
        Draw $\theta_{n,r}^{(t)} \sim \text{Unif}(0,1)$\;
        $p_{n,r}^{(t)} \leftarrow \frac{1}{r} \sum_{j=1}^{r} \left( \mathbf{1}_{\kappa^{(t)}_{n,r,j} > \kappa^{(t)}_{n,r,r}} + \theta_{n,r} \cdot \mathbf{1}_{\kappa^{(t)}_{n,r,j} = \kappa^{(t)}_{n,r,r}} \right)$\;
    }
    \textbf{end}\;

    \For{$r \in (n, \ldots, t+1)$}{
        \For{$j \in (r, \ldots, n)$}{
            $\kappa^{(t)}_{n,r,j} \leftarrow s^{(1)}_{n,r,t}\!\left(X_{n,j}\,;\,[X_{n,r+1}, \ldots, X_{n,n}],\,(X_{n,1}, \ldots, X_{n,t})\right)$\;
        }
        \textbf{end}\;
        Draw $\theta_{n,r}^{(t)} \sim \text{Unif}(0,1)$\;
        $p_{n,r}^{(t)} \leftarrow \frac{1}{n - r + 1} \sum_{j=r}^{n} \left( \mathbf{1}_{\kappa^{(t)}_{n,r,j} > \kappa^{(t)}_{n,r,r}} + \theta_{n,r} \cdot \mathbf{1}_{\kappa^{(t)}_{n,r,j} = \kappa^{(t)}_{n,r,r}} \right)$\;
    }
    \textbf{end}\;

    $\widehat{F}^{(0)}_{n,t}(z) := \frac{1}{t} \sum_{r=1}^{t} \mathbf{1}_{p_{n,r}^{(t)} \leq z}$; \quad
    $\widehat{F}^{(1)}_{n,t}(z) := \frac{1}{n - t} \sum_{r=t+1}^{n} \mathbf{1}_{p_{n,r}^{(t)} \leq z}$\;

    $W_{n,t}^{(0)} \leftarrow \sqrt{t} \cdot \text{KS}(\widehat{F}^{(0)}_{n,t}, u)$; \quad
    $W_{n,t}^{(1)} \leftarrow \sqrt{n - t} \cdot \text{KS}(\widehat{F}^{(1)}_{n,t}, u)$\;

    $p_{n,t}^\text{left} = 1 - F_t(W_{n,t}^{(0)})$; \quad $p_{n,t}^\text{right} = 1 - F_{n-t}(W_{n,t}^{(1)})$\;

    $\big($$F_n$ is the CDF of $\sqrt{n}\cdot \text{KS}(\widehat{F}_n,u)$ where $\widehat{F}_n$ is the empirical CDF of $n$ iid observation from $U(0,1)$ $\big)$\;

    $p^{\texttt{CONF}}_{n,t} = \min \{2p^\text{left}_{n,t}, 2p^\text{right}_{n,t}, 1\}$\;
}
\textbf{end}\;

$\widehat{\tau}_n \leftarrow \arg\max_{t \in [n-1]} p^{\texttt{CONF}}_{n,t}$\;
$C^{\texttt{MCP}}_{n,1-\alpha} \leftarrow \{\,t \in [n-1] : p^{\texttt{CONF}}_{n,t} > \alpha\,\}$\;
\Return $\widehat{\tau}_n$ and $C^{\texttt{MCP}}_{n,1-\alpha}$\;
\end{algorithm2e}

\newpage
\section{Auxiliary Lemmas and Proof of Results}
\label{Appendix: Proofs}

\begin{applemma}
\label{lemma1}
    Suppose $X_1,\cdots,X_\tau \iid R$ and $X_{\tau+1},\cdots,X_n \iid Q$ where $Q$ is absolutely continuous with respect to $R$ and $\frac{dQ}{dR}(x) < \infty$ for all $x \in \R$. Let, $\widehat{\mathcal{F}}_n = \frac{1}{n}\sum_{i=1}^n \delta_{X_i}$
    denote the empirical distribution of all the data points. Then, $X_n | \widehat{\mathcal{F}}_n \sim \sum_{i=1}^n w_i\delta_{X_i},$
    where $w_i = \frac{\frac{dQ}{dR}(X_i)}{\sum_{j=1}^n \frac{dQ}{dR}(X_j)}, \hspace{0.3cm} i \in [n].$
\end{applemma}
\begin{pf}
    Observe that it suffices to show, for all $A, B \in \mathcal{B}(\R)$,
\[
\P(X_n \in A, \widehat{\mathcal{F}}_n \in B) = \E \left[ \sum_{i=1}^{n} w_i \mathbf{1}(Z_i \in A) \mathbf{1}(\widehat{\mathcal{F}}_n \in B) \right].\tag{*}
\]

We use the following property of Radon–Nikodym derivative: that for any measurable function $h : \R^{n+1} \to \R$, we have
\begin{align*}
\E_{(X_1, \dots, X_n) \sim R^{n-1} \times Q} \left[ h(X_1, \dots, X_n) \right] 
&= \E_{(X_1, \dots, X_n) \sim R^n} \left[ h(X_1, \dots, X_n) \frac{dQ}{dR}(X_n) \right]. \tag{**}
\end{align*}

Now to show (*), observe:
\begin{align*}
\P(X_n \in A, \widehat{\mathcal{F}}_n \in B) 
&= \E_{R^{\tau} \times Q^{n - \tau}} \left[ \mathbf{1}(X_n \in A) \mathbf{1}(\widehat{\mathcal{F}}_n \in B) \right] \\
&= \E_{R^n} \left[ \frac{dQ}{dR}(X_{\tau+1}) \cdots \frac{dQ}{dR}(X_n) \mathbf{1}(X_n \in A) \mathbf{1}(\widehat{\mathcal{F}}_n \in B) \right] \tag{by (**) } \\
&= \E_{R^n} \left[ \frac{dQ}{dR}(X_{\tau+1}) \cdots \frac{dQ}{dR}(X_{n-1}) w_n \bigg(\sum_{j \in [n]} \frac{dQ}{dR}(X_{j}) \bigg) \mathbf{1}(X_n \in A) \mathbf{1}(\widehat{\mathcal{F}}_n \in B) \right] \\
&= \sum_{j \leqslant \tau} \E_{R^n} \left[ \frac{dQ}{dR}(X_{\tau+1}) \cdots \frac{dQ}{dR}(X_{n-1}) w_n \frac{dQ}{dR}(X_j) \mathbf{1}(X_n \in A) \mathbf{1}(\widehat{\mathcal{F}}_n \in B)  \right] \\
&\quad + \sum_{j = \tau+1}^{n} \E_{R^n} \left[ \frac{dQ}{dR}(X_{\tau+1}) \cdots \frac{dQ}{dR}(X_{n-1}) w_n \frac{dQ}{dR}(X_j) \mathbf{1}(X_n \in A) \mathbf{1}(\widehat{\mathcal{F}}_n \in B) \right] \\
&= \sum_{j \leq \tau} \E_{R^n} \left[ \frac{dQ}{dR}(X_{\tau+1}) \cdots \frac{dQ}{dR}(X_{n-1}) w_j \frac{dQ}{dR}(X_n) \mathbf{1}(X_j \in A) \mathbf{1}(\widehat{\mathcal{F}}_n \in B) \right] \\
&\quad + \sum_{j > \tau} \E_{R^n} \left[ \frac{dQ}{dR}(X_{\tau+1}) \cdots \frac{dQ}{dR}(X_{j-1}) w_j \frac{dQ}{dR}(X_n) \mathbf{1}(X_j \in A) \mathbf{1}(\widehat{\mathcal{F}}_n \in B) \right] \\
&\text{[since the expectations are w.r.t. i.i.d. variables, so we can reindex]} \\
&= \sum_{j \in [n]} \E_{R^n} \left[ \frac{dQ}{dR}(X_{\tau+1}) \cdots \frac{dQ}{dR}(X_n) w_j \mathbf{1}(X_j \in A) \mathbf{1}(\widehat{\mathcal{F}}_n \in B) \right] \\
&= \sum_{j \in [n]} \E_{R^{\tau} \times Q^{n - \tau}} \left[ w_j \mathbf{1}(X_j \in A) \mathbf{1}(\widehat{\mathcal{F}}_n \in B) \right] \tag{by the properties of Radon–Nikodym derivative} \\
&= \E \left[ \sum_{j \in [n]} w_j \mathbf{1}(X_j \in A) \mathbf{1}(\widehat{\mathcal{F}}_n \in B) \right],
\end{align*}

as desired. Hence, we have the lemma. \qed
\end{pf}

\begin{applemma}
\label{pval dis}
    $y_1 < \cdots < y_k$ are reals. Let $Y$ be a random variable taking values $y_1, \dots, y_k$ such that
\[
\P(Y = y_i) = p_i, \quad i \in [k], \quad p_i \in (0,1),
\]
$\tau$ is independent of $Y$ and $\tau \sim U(0,1)$.
Define
\[
P = \sum_{i=1}^{k} p_i \mathbf{1}(y_i > Y) + \tau \sum_{i=1}^{k} p_i \mathbf{1}(y_i = Y).
\]
Then \( P \sim U(0,1) \).
\end{applemma}

\begin{pf}
    For $\epsilon \in (0,1)$,
\begin{align*}
\P(P \leq \epsilon) 
&= \E\left[ \P(P \leq \epsilon \mid \tau) \right] \\
&= \E\left( \sum_{i=1}^{k} \P(P \leq \epsilon, Y = y_i \mid \tau) \right) \\
&= \E \left( \sum_{i=1}^{k-1} p_i \mathbf{1} \left( \sum_{j=i+1}^{k} p_j + \tau p_i \leq \epsilon \right) 
+ p_k \mathbf{1}(\tau p_k \leq \epsilon) \right) \\
&= \sum_{i=1}^{k-1} p_i \P\left( \tau \leq \frac{\epsilon - \sum_{j=i+1}^{k} p_j}{p_i} \right) 
+ p_k \P\left( \tau \leq \frac{\epsilon}{p_k} \right).
\end{align*}

Now,
\[
\P\left( \tau \leq \frac{\epsilon - \sum_{j=i+1}^{k} p_j}{p_i} \right) =
\begin{cases}
0 & \text{if } \epsilon < \sum_{j=i+1}^{k} p_j, \\
1 & \text{if } \epsilon \geq \sum_{j=i}^{k} p_j, \\
\frac{\epsilon - \sum_{j=i+1}^{k} p_j}{p_i} & \text{otherwise}.
\end{cases}
\]

Hence, for some $i_0 \in [k]$,
\[
\P(P \leq \epsilon) 
= \frac{\epsilon - \sum_{j=i_0+1}^{k} p_j}{p_{i_0}} \cdot p_{i_0} + \sum_{j=i_0+1}^{k} p_j 
= \epsilon.
\]

Hence \( P \sim U(0,1) \). \qed
\end{pf}

\begin{applemma}
\label{lemma3}
     Let $y_1 < \cdots < y_k$ be real numbers. Let $Y$ be a random variable taking values $y_1, \dots, y_k$ such that
\[
\P(Y = y_i) = p_i, \quad i \in [k], \quad p_i \in (0,1).
\]
Let $\tau \sim U(0,1)$ be independent of $Y$. Define
\[
P = \sum_{i=1}^{k} \alpha_i \mathbf{1}(y_i > Y) + \tau \sum_{i=1}^{k} \alpha_i \mathbf{1}(y_i = Y),
\]
where $\alpha_i \in (0,1), i \in [k]$ are such that $\sum_{i=1}^{k} \alpha_i = 1$.

Let $F$ be the CDF of $P$ and $F_0$ be the CDF of $U(0,1)$ random variable. Then
\[
\sup_{\epsilon \in \mathbb{R}} |F(\epsilon) - F_0(\epsilon)| = \max_{1 \leq i \leq k} \left| \sum_{j=1}^{i} (p_j - \alpha_j) \right|.
\]
\end{applemma}

\begin{pf}
    By a similar calculation as in Lemma~\ref{pval dis}, we have for $\epsilon \in (0,1)$,
\begin{align*}
F(\epsilon) &= \P(P \leq \epsilon) \\
&= \sum_{i=1}^{k-1} p_i \P\left( \tau \leq \frac{\epsilon - \sum_{j=i+1}^{k} \alpha_j}{\alpha_i} \right) + p_k \P\left( \tau \leq \frac{\epsilon}{\alpha_k} \right).
\end{align*}

Now,
\[
\P\left( \tau \leq \frac{\epsilon - \sum_{j=i+1}^{k} \alpha_j}{\alpha_i} \right) =
\begin{cases}
0 & \text{if } \epsilon < \sum_{j=i+1}^{k} \alpha_j, \\
1 & \text{if } \epsilon \geq \sum_{j=i}^{k} \alpha_j, \\
\frac{\epsilon - \sum_{j=i+1}^{k} \alpha_j}{\alpha_i} & \text{otherwise}.
\end{cases}
\]

Therefore,
\[
F(\epsilon) = p_{i_0} \cdot \frac{\epsilon - \sum_{j=i_0+1}^{k} \alpha_j}{\alpha_{i_0}} + \sum_{j=i_0+1}^{k} p_j, \quad \text{where } i_0 \text{ is such that } \epsilon \in \bigg[ \sum_{j=i_0+1}^{k} \alpha_j, \sum_{j=i_0}^{k} \alpha_j \bigg).
\]

Hence,
\begin{align*}
\sup_{\epsilon \in \mathbb{R}} |F(\epsilon) - F_0(\epsilon)| 
&= \sup_{\epsilon \in (0,1)} |F(\epsilon) - \epsilon| \\
&= \max_{0 \leq i_0 \leq k-1} \sup_{\epsilon \in [ \sum_{j=i_0+1}^{k} \alpha_j, \sum_{j=i_0}^{k} \alpha_j )} \left| \frac{p_{i_0}}{\alpha_{i_0}} \left( \epsilon - \sum_{j=i_0+1}^{k} \alpha_j \right) + \sum_{j=i_0+1}^{k} p_j - \epsilon \right|.
\end{align*}

Now,
\begin{align*}
\sup_{\epsilon \in [ \sum_{j=i_0+1}^{k} \alpha_j, \sum_{j=i_0}^{k} \alpha_j )}
\left| \frac{p_{i_0}}{\alpha_{i_0}} \left( \epsilon - \sum_{j=i_0+1}^{k} \alpha_j \right) + \sum_{j=i_0+1}^{k} p_j - \epsilon \right|
\end{align*}

\begin{align*}
= \max \Bigg\{
\left| \sum_{j=i_0+1}^{k} (p_j - \alpha_j) \right|, \quad
\left| \sum_{j=i_0}^{k} (p_j - \alpha_j) \right|
\Bigg\}.
\end{align*}

Hence,
\begin{align*}
\sup_{\epsilon \in \mathbb{R}} |F(\epsilon) - F_0(\epsilon)| 
&= \max_{0 \leq i_0 \leq k-1} \max \left\{ \left| \sum_{j=i_0+1}^{k} (p_j - \alpha_j) \right|, \left| \sum_{j=i_0}^{k} (p_j - \alpha_j) \right| \right\} \\
&= \max_{0 \leq i_0 \leq k-1} \left| \sum_{j=i_0+1}^{k} (p_j - \alpha_j) \right| \\
& = \max_{1 \leq i \leq k} \left| \sum_{j=1}^{i} (p_j - \alpha_j) \right|.
\end{align*}

as desired. Hence, the lemma follows. \qed

\end{pf}

\textbf{Proof of \Cref{th2}:} By Lemma~\ref{lemma1}, the conditional distribution of $X_n$ given $\{X_1,\cdots,X_n\}$ is
    \[ X_n | \{X_1,\cdots,X_n\} \sim \sum_{i=1}^n w_i\delta_{X_i}, \]
    where \[ w_i = \frac{\frac{dQ}{dR}(X_i)}{\sum_{j=1}^n \frac{dQ}{dR}(X_j)}, \hspace{0.3cm} i \in [n].\]
    Therefore, in Lemma~\ref{lemma3} setting,
    \[ \alpha_i = \frac{1}{n},y_i = X_{(i)}, p_i =  \frac{\frac{dQ}{dR}(X_{(i)})}{\sum_{j=1}^n \frac{dQ}{dR}(X_j)} \hspace{0.2cm} i \in [n], \] 
    yields the theorem. \qed
\\
\medskip
\\
\textbf{Proof of \Cref{Th:th3}:} 
    By \Cref{assn1}, there exists $\delta \in (0,1)$ such that the $\delta$-th quantile of the mixture distribution $cR + (1-c)Q$ is unique. With this $\delta$, we show the following claim, which implies the theorem.

\textbf{Claim 1:} There exists $\mu \neq 1$, $\mu > 0$ such that
\begin{equation}
\label{eq:claim_ratio}
\frac{\frac{1}{\delta n} \sum_{j \leqslant \delta n} \frac{dQ}{dR}(X_{n,(j)})}{\frac{1}{n} \sum_{j \in [n]} \frac{dQ}{dR}(X_{n,(j)})} \xrightarrow[]{\text{a.s.}} \mu,
\end{equation}
as $n \to \infty$, where $X_{n,(1)} \leq \dots \leq X_{n,(n)}$ are the order statistics of $\{X_{n,j}\}_{j=1}^n$.

From the claim, it follows that:
\begin{align*}
\left| \frac{\sum_{j \leqslant \delta n} \frac{dQ}{dR}(X_{n,(j)})}{\sum_{j \in [n]} \frac{dQ}{dR}(X_{n,(j)})} - \frac{\delta n}{n} \right| 
&= \delta \left| \frac{\frac{1}{\delta n} \sum_{j \leqslant \delta n} \frac{dQ}{dR}(X_{n,(j)})}{\frac{1}{n} \sum_{j \in [n]} \frac{dQ}{dR}(X_{n,(j)})} - 1 \right| \\
&\xrightarrow[]{\text{a.s.}} \delta |\mu - 1| > 0. \tag{*}
\end{align*}

From \Cref{th2}, we have:
\begin{align*}
\sup_{t \in \mathbb{R}} \left| F_{p_n}(t) - F(t) \right| 
&= \max_{1 \leq i \leq n} \left| \frac{\sum_{j \leqslant i} \frac{dQ}{dR}(X_{n,(j)})}{\sum_{j \in [n]} \frac{dQ}{dR}(X_{n,(j)})} - \frac{i}{n} \right|.
\end{align*}

Hence, it follows that:
\begin{align*}
\liminf_{n \to \infty} \sup_{t \in \mathbb{R}} \left| F_{p_n}(t) - F(t) \right| 
&= \liminf_{n \to \infty} \max_{i \in [n]} \left| \frac{\sum_{j \leqslant i} \frac{dQ}{dR}(X_{n,(j)})}{\sum_{j \in [n]} \frac{dQ}{dR}(X_{n,(j)})} - \frac{i}{n} \right| \\
&\geq \liminf_{n \to \infty} \left| \frac{\sum_{j \leqslant \delta n} \frac{dQ}{dR}(X_{n,(j)})}{\sum_{j \in [n]} \frac{dQ}{dR}(X_{n,(j)})} - \frac{\delta n}{n} \right| \\
&\xrightarrow[]{\text{a.s.}} \delta |\mu - 1| > 0.
\end{align*}

So, upon proving the claim, the theorem follows.

\textbf{Proof of the Claim 1:} For the denominator of the expression in ~\eqref{eq:claim_ratio}, observe:
\begin{align*}
\frac{1}{n} \sum_{j \in [n]} \frac{dQ}{dR}(X_{n,j})
&= \frac{\tau_n}{n} \cdot \frac{1}{\tau_n} \sum_{j \leqslant \tau_n} \frac{dQ}{dR}(X_{n,j})
+ \left(1 - \frac{\tau_n}{n}\right) \cdot \frac{1}{n - \tau_n} \sum_{\tau_n < j \leqslant n} \frac{dQ}{dR}(X_{n,j}) \\
&\xrightarrow{\text{a.s.}} c \cdot \E_{Z \sim R} \left[ \frac{dQ}{dR}(Z) \right]
+ (1 - c) \cdot \E_{Z \sim Q} \left[ \frac{dQ}{dR}(Z) \right].
\end{align*}

Now, with $\iota_n = \delta n$, 
\begin{align*}
\frac{1}{\iota_n} \sum_{j \leqslant \iota_n} \frac{dQ}{dR}(X_{n,(j)})
&= \frac{1}{\iota_n} \left[
\sum_{j \leqslant \tau_n} \frac{dQ}{dR}(X_{n,j}) \mathbf{1}\{X_{n,j} \text{ has rank } \leqslant \iota_n\} \right. \\
&\qquad \left. + \sum_{j > \tau_n} \frac{dQ}{dR}(X_{n,j}) \mathbf{1}\{X_{n,j} \text{ has rank } \leqslant \iota_n\}
\right], \tag{**}
\end{align*}
where by rank in $X_{n,j}$ we mean rank in $\{X_{n,1},\cdots,X_{n,n}\}$. Define
\begin{align*}
R_n &= \sum_{j \leqslant \tau_n} \mathbf{1}\{X_{n,j} \text{ has rank } \leqslant \iota_n\} \\
&= \sum_{j \leqslant \tau_n} \mathbf{1}\{X_{n,j} \leqslant X_{n,(\iota_n)}\} \\
&= \tau_n \cdot \widehat{F}_{\tau_n}(X_{n,(\iota_n)}),
\end{align*}
where $\widehat{F}_{\tau_n}(\cdot)$ is the empirical CDF of $\{X_{n,1}, \dots, X_{n,\tau_n}\}$.

Let $q_{\delta,c}$ denote the $\delta$-th quantile of $cR + (1-c)Q$. Since $q_{\delta,c}$ is unique (by assumption), we have
\begin{equation}
X_{n,(\iota_n)} \xrightarrow{\text{a.s.}} q_{\delta,c}, \quad \text{as } n \to \infty.
\end{equation}
By the Glivenko–Cantelli lemma:
\[
\sup_{t \in \mathbb{R}} \left| \widehat{F}_{\tau_n}(t) - R(t) \right| \xrightarrow{\text{a.s.}} 0.
\]
So we can write
\begin{align*}
\widehat{F}_{\tau_n}(X_{n,(\iota_n)}) &= \widehat{F}_{\tau_n}(X_{n,(\iota_n)}) - R(X_{n,(\iota_n)}) + R(X_{n,(\iota_n)}) - R(q_{\delta,c}) + R(q_{\delta,c}),
\end{align*}
to have 
\begin{align*}
    |\widehat{F}_{\tau_n}(X_{n,(\iota_n)}) - R(q_{\delta,c})| &\leqslant \sup_{t \in \R}|\widehat{F}_{\tau_n}(X_{n,(\iota_n)}) - R(t)| + |R(X_{n,(\iota_n)}) - R(q_{\delta,c})| \\
    &= o(1) \hspace{0.3cm}\text{a.s.}.   
\end{align*}
Thus, 
\begin{equation}
\label{eq:ecdf_conv}
\widehat{F}_{\tau_n}(X_{n,(\iota_n)}) \xrightarrow{\text{a.s.}} R(q_{\delta,c}) \quad \Rightarrow \quad R_n \xrightarrow{\text{a.s.}} \infty.
\end{equation}
Observe that, we can relabel the indices in (**) to write:
\begin{align*}
&\quad\frac{1}{\iota_n} \sum_{j \leqslant \iota_n} \frac{dQ}{dR}(X_{n,(j)}) \\
&= \frac{1}{\iota_n} \left[ \sum_{j \leqslant R_n} \frac{dQ}{dR}(X_{n,j})
+ \sum_{\tau_n < j \leqslant \tau_n + (\iota_n - R_n)} \frac{dQ}{dR}(X_{n,j}) \right] \\
&= \frac{R_n}{\iota_n} \cdot \frac{1}{R_n} \sum_{j \leqslant R_n} \frac{dQ}{dR}(X_{n,j})  + \frac{\iota_n - R_n}{\iota_n} \cdot \frac{1}{\iota_n - R_n} \sum_{\tau_n < j \leqslant \tau_n + (\iota_n - R_n)} \frac{dQ}{dR}(X_{n,j}). \tag{***}
\end{align*}
Now by the strong law of large numbers:
\begin{align*}
\frac{1}{R_n} \sum_{j \leqslant R_n} \frac{dQ}{dR}(X_{n,j}) &\xrightarrow{\text{a.s.}} \E_{Z \sim R} \left[ \frac{dQ}{dR}(Z) \right], \\
\frac{1}{\iota_n - R_n} \sum_{\tau_n < j \leqslant \tau_n + (\iota_n - R_n)} \frac{dQ}{dR}(X_{n,j}) &\xrightarrow{\text{a.s.}} \E_{Z \sim Q} \left[ \frac{dQ}{dR}(Z) \right].
\end{align*}
Furthermore, 
\begin{align*}
\frac{R_n}{\iota_n} &= \frac{\tau_n}{\iota_n} \widehat{F}_{\tau_n}(X_{n,(\iota_n)}) = \frac{c \cdot n}{\delta \cdot n} \widehat{F}_{\tau_n}(X_{n,(\iota_n)}) \overset{a.s.}{\longrightarrow} \frac{c}{\delta} R(q_{\delta, c}).
\end{align*}
Hence from (***), we have
\begin{align*}
\frac{1}{\iota_n} \sum_{j \leqslant \iota_n} \frac{dQ}{dR}(X_{n,(j)}) 
&\overset{a.s.}{\longrightarrow} 
\frac{c}{\delta} R(q_{\delta, c}) \cdot \E_{Z \sim R} \left[ \frac{dQ}{dR}(Z) \right]  + \left(1 - \frac{c}{\delta} R(q_{\delta, c}) \right) \cdot \E_{Z \sim Q} \left[ \frac{dQ}{dR}(Z) \right].
\end{align*}
Hence,
\begin{equation}
\frac{\frac{1}{\delta n} \sum_{j \leqslant \delta n} \frac{dQ}{dR}(X_{n,(j)})}
{\frac{1}{n} \sum_{j \in [n]} \frac{dQ}{dR}(X_{n,(j)})}
\overset{a.s.}{\longrightarrow} \mu,
\label{eq:mu_limit}
\end{equation}
with
\begin{align*}
\mu &= \frac{\frac{c}{\delta} R(q_{\delta, c}) \cdot \E_{Z \sim R} \left( \frac{dQ}{dR}(Z) \right) 
+ \left(1 - \frac{c}{\delta} R(q_{\delta, c}) \right) \cdot \E_{Z \sim Q} \left( \frac{dQ}{dR}(Z) \right)}
{c \cdot \E_{Z \sim R} \left[ \frac{dQ}{dR}(Z) \right] + (1 - c) \cdot \E_{Z \sim Q} \left[ \frac{dQ}{dR}(Z) \right]}.
\end{align*}
For the choice of $c, \delta$, we have $\mu \ne 1$. 

Hence, we have the claim and the theorem thereafter.\qed
\\
\medskip
\\
\textbf{Proof of \Cref{th: Th diff pval}:} 
From the proof of Lemma~\ref{lemma3}, we have for $\epsilon \in \left[\frac{i}{n}, \frac{i+1}{n}\right)$, $0 \le i \le n-1$:
\begin{equation}
F_{p_n}(\epsilon) = 
\frac{n \cdot \frac{dQ}{dR}(X_{n,(n-i)})}{\sum\limits_{j \in [n]} \frac{dQ}{dR}(X_{n,j})} \left(\epsilon - \frac{i}{n} \right) 
+ \frac{ \sum\limits_{j=n-i+1}^{n} \frac{dQ}{dR}(X_{n,(j)}) }{ \sum\limits_{j \in [n]} \frac{dQ}{dR}(X_{n,j}) }.
\label{eq:star}
\end{equation}
Thus, $F_{p_n}$ is piecewise linear.
Hence,
\begin{align*}
\sup_{\epsilon \in [0,1]} \left| F_{p_n}(\epsilon) - F_{p_{n+1}}(\epsilon) \right| 
&= \max_{\substack{\epsilon \in \left\{ \frac{i}{n}, \frac{k}{n+1} \right\} \\ 1 \le i \le n, 1 \le k \le n+1}} \left| F_{p_n}(\epsilon) - F_{p_{n+1}}(\epsilon) \right|.
\end{align*}
Now from \eqref{eq:star}, we have
\begin{align*}
F_{p_n} \left( \frac{i}{n} \right) 
&= \frac{ \sum\limits_{j=n-i+1}^{n} \frac{dQ}{dR}(X_{n,(j)}) }{ \sum\limits_{j \in [n]} \frac{dQ}{dR}(X_{n,j}) }, \\
F_{p_{n+1}} \left( \frac{k}{n+1} \right) 
&= \frac{ \sum\limits_{j=n-k+2}^{n+1} \frac{dQ}{dR}(X_{n+1,(j)}) }{ \sum\limits_{j \in [n+1]} \frac{dQ}{dR}(X_{n+1,j}) }.
\end{align*}
Also, with $k = \left\lfloor i \left(1 + \frac{1}{n} \right) \right\rfloor$, we have
\[
\frac{k}{n+1} \le \frac{i}{n} < \frac{k+1}{n+1}.
\]
Using that,
\begin{equation}
F_{p_{n+1}} \left( \frac{i}{n} \right) = 
\frac{(n+1) \cdot \frac{dQ}{dR}(X_{n+1,(n+1-k)})}{\sum\limits_{j \in [n+1]} \frac{dQ}{dR}(X_{n+1,j})} 
\left( \frac{i}{n} - \frac{k}{n+1} \right)
+ \frac{ \sum\limits_{j=n-k+2}^{n+1} \frac{dQ}{dR}(X_{n+1,(j)}) }{ \sum\limits_{j \in [n+1]} \frac{dQ}{dR}(X_{n+1,j}) }.
\end{equation}
Hence we have 
\begin{align}
\label{eq:double_star}
  &\quad F_{p_n} \left( \frac{i}{n} \right) -  F_{p_{n+1}} \left( \frac{i}{n} \right) \nonumber \\
  &=  \frac{ \sum\limits_{j=n-i+1}^{n} \frac{dQ}{dR}(X_{n,(j)}) }{ \sum\limits_{j \in [n]} \frac{dQ}{dR}(X_{n,j}) } - \frac{(n+1) \cdot \frac{dQ}{dR}(X_{n+1,(n+1-k)})}{\sum\limits_{j \in [n+1]} \frac{dQ}{dR}(X_{n+1,j})} 
\left( \frac{i}{n} - \frac{k}{n+1} \right) 
&- \frac{ \sum\limits_{j=n-k+2}^{n+1} \frac{dQ}{dR}(X_{n+1,(j)}) }{ \sum\limits_{j \in [n+1]} \frac{dQ}{dR}(X_{n+1,j}) }.
\end{align}
Now, to show the stochastic boundedness as per the goal, we use the following fact: 
\begin{fact}
    Let $\{Z_n\}_{n \in \N}$ be a sequence of random variables that satisfies every subsequence of $\{Z_n\}_{n \in \N}$ has a further subsequence which is $O_p(1)$. Then $\{Z_n\}_{n \in \N}$ is $O_p(1)$.
\end{fact}

Let $\{n_{\Tilde{k}}\}_{\Tilde{k} \in \mathbb{N}}$ be any subsequence of naturals and $\{i_{\Tilde{k}}\}_{\Tilde{k} \in \mathbb{N}}$ be a sequence such that $1 \le i_{\Tilde{k}} \le n_{\Tilde{k}}$. Since $0 < \frac{i_{\Tilde{k}}}{n_{\Tilde{k}}} < 1$, the sequence is bounded, and hence $\left\{ \frac{i_{\Tilde{k}}}{n_{\Tilde{k}}} \right\}$ has a convergent subsequence. WLOG, let $\frac{i_{\Tilde{k}_m}}{n_{\Tilde{k}_m}} \to \delta \in (0,1)$ and $n_{\Tilde{k}_m} \to \infty$.

From the proof of \cref{Th:th3}, we have
\[
\frac{1}{n+1} \sum_{j \in [n+1]} \frac{dQ}{dR}(X_{n+1,j}) \xrightarrow{a.s.} u,
\]
for some $u \in \R$ hence, 
\[ \frac{1}{n+1} \sum_{j \in [n+1]} \frac{dQ}{dR}(X_{n+1,j}) = O_p(1). \]

Also we have $|\frac{i}{n} - \frac{k}{n+1}| \leq \frac{1}{n+1}$ by the choice of $k$. Furthermore, since $\frac{i_{\Tilde{k}_m}}{n_{\Tilde{k}_m}} \to \delta$ and by uniqueness of the $(1-\delta)$-th quantile of $cR + (1-c)Q$, let $\Tilde{q}$ be that quantile, then
\[
X_{n_{\Tilde{k}_m}+1, (n_{\Tilde{k}_m}-i_{\Tilde{k}_m})} \xrightarrow{a.s.} \Tilde{q}.
\]
Now $Q,R$ has densities $q$, $r$ which are continuous, so 
\[
\frac{dQ}{dR}(z) = \frac{q(z)}{r(z)}
\]
is continuous at $\Tilde{q}$ and hence
\[
\frac{dQ}{dR}\left(X_{n_{\Tilde{k}_m}+1, (n_{\Tilde{k}_m}-i_{\Tilde{k}_m})} \right) \xrightarrow{a.s.} \frac{q(\Tilde{q})}{r(\Tilde{q})}.
\]
Hence we have along the sequence $\{n_{\Tilde{k}_m}\}_{m \in \N}$,
\begin{align*}
\sqrt{n} \left|\frac{(n+1) \frac{dQ}{dR}(X_{n+1,(n-k)}) \left[\frac{i}{n} - \frac{k}{n+1}\right]}
{\sum_{j \in [n+1]} \frac{dQ}{dR}(X_{n+1,j})}\right|   \leq \frac{
    \frac{dQ}{dR}(X_{n+1,(n-k)})
}{
    \frac{1}{n+1} \sum_{j \in [n+1]} \frac{dQ}{dR}(X_{n+1,j})
}\cdot 
\frac{1}{\sqrt{n+1}} = o_p(1).
\end{align*}
Also, from the proof of \Cref{Th:th3}, since 
\[
 \quad \mathbb{E}_{Z \sim Q} \left[ \left( \frac{dQ}{dR}(Z) \right)^2 \right] < \infty,
\quad
\mathbb{E}_{Z \sim R} \left[ \left( \frac{dQ}{dR}(Z) \right)^2 \right] < \infty,
\]
We can see using Lyapunov's Central Limit Theorem, 
\[
\sqrt{n_{\Tilde{k}_m}} \left(
\begin{bmatrix}
\sum\limits_{j = n_{\Tilde{k}_m} - i_{\Tilde{k}_m} + 1}^{n_{\Tilde{k}_m}} \frac{dQ}{dR}(X_{n_{\Tilde{k}_m}, j}) \\
\sum\limits_{j \in [n_{\Tilde{k}_m}]} \frac{dQ}{dR}(X_{n_{\Tilde{k}_m}, j})
\end{bmatrix}
- \Tilde{\mu}
\right)
\overset{d}{\longrightarrow} \mathcal{N}(0, \Sigma),
\]
for some \(\mu, \Sigma\). Thus along the sequence \(\{n_{\Tilde{k}_m}\}_{m \in \mathbb{N}}\),
\[
\sqrt{n} \left[
\frac{\sum_{j = n - i + 1}^{n} \frac{dQ}{dR}(X_{n,(j)})}{\sum_{j \in [n]} \frac{dQ}{dR}(X_{n,j})}
-
\frac{\sum_{j = n - k + 2}^{n+1} \frac{dQ}{dR}(X_{n+1,(j)})}{\sum_{j \in [n+1]} \frac{dQ}{dR}(X_{n+1,j})}
\right]
= O_p(1).
\]
Thus, from \eqref{eq:double_star}, we obtain
\[
\sup_{\epsilon \in \mathbb{R}} \sqrt{n} \left| F_{p_n}(\epsilon) - F_{p_{n+1}}(\epsilon) \right| = O_p(1),
\]
Completing the proof. \hfill \qed
\\
\medskip
\\
\textbf{Proof of \Cref{th: Main theorem}:} Observe that, without loss of generality, we can assume $\Tilde{c} = 1$ and prove the theorem for that. Then the statement for any $\Tilde{c}$ follows thereafter. Let $F$ be the CDF of a $U(0,1)$ random variable. Then,
\[
\sup_{t \in [0,1]} \left| \widehat{F}_{n, \lfloor cn + n^{1/4} \rfloor}(t) - t \right| = \left\| \widehat{F}_{n, \lfloor cn + n^{1/4} \rfloor} - F \right\|_{\infty}.
\]

Let $\mathcal{G}_{n,j} = \sigma\left( \{ X_{n,k} \}_{1 \leq k \leq j} \right)$ and $\P_{\mathcal{G}_{n,j}}(\cdot)$ denote the conditional probability conditioned on $\mathcal{G}_{n,j}$. Then by the triangle inequality, we have:
\begin{align*}
\left\| \widehat{F}_{n, \lfloor cn + n^{1/4} \rfloor}(t) - F(t) \right\|_{\infty} 
&\geq \left\| \frac{1}{\lfloor cn + n^{1/4} \rfloor} \sum_{k=1}^{\lfloor cn + n^{1/4} \rfloor} \P_{\mathcal{G}_{n,k}}(p_{n,k} \leq t) - F(t) \right\|_{\infty} \\
&\quad - \left\| \widehat{F}_{n, \lfloor cn + n^{1/4} \rfloor}(t) - \frac{1}{\lfloor cn + n^{1/4} \rfloor} \sum_{k=1}^{\lfloor cn + n^{1/4} \rfloor} \P_{\mathcal{G}_{n,k}}(p_{n,k} \leq t) \right\|_{\infty}.
\end{align*}

We prove the theorem by showing the following steps:

\textbf{Step 1:} Show that
\[
\liminf_{n \to \infty} \left\| \frac{1}{\lfloor cn + n^{1/4} \rfloor} \sum_{k=1}^{\lfloor cn + n^{1/4} \rfloor} \P_{\mathcal{G}_{n,k}}(p_{n,k} \leq t) - F(t) \right\|_{\infty} > 0 \quad \text{a.s.}.
\]

\textbf{Step 2:} Show that
\[
\left\| \widehat{F}_{n, \lfloor cn + n^{1/4} \rfloor}(t) - \frac{1}{\lfloor cn + n^{1/4} \rfloor} \sum_{k=1}^{\lfloor cn + n^{1/4} \rfloor} \P_{\mathcal{G}_{n,k}}(p_{n,k} \leq t) \right\|_{\infty} = o(1) \quad \text{a.s.}.
\]

\textbf{Proof of Step 1:} For brevity, write
\[
F_{n,k}(t) := \P_{\mathcal{G}_{n,k}}(p_{n,k} \leq t),
\]

which is the conditional CDF of $p_{n,k}$ given $\mathcal{G}_{n,k}$. So we have to show
\begin{equation}
\liminf_{n \to \infty} \sup_{t \in [0,1]} \left| \frac{1}{\lfloor cn + n^{1/4} \rfloor} \sum_{k=1}^{\lfloor cn + n^{1/4} \rfloor} F_{n,k}(t) - t \right| > 0 \quad \text{a.s.}
\label{eq:star}
\end{equation}
We will show
\[
\liminf_{n \to \infty} \bigg|  \frac{1}{\lfloor cn + n^{1/4} \rfloor} \sum_{k=1}^{\lfloor cn + n^{1/4} \rfloor} F_{n,k}\left( \frac{1}{2} \right) - \frac{1}{2} \bigg| > 0 \quad \text{a.s.},
\]

which will imply \eqref{eq:star}. Now observe that for $n$ even,
\[
F_{n,n}\left( \frac{1}{2} \right) = \frac{\sum\limits_{j = \frac{n}{2} + 1}^{n} \frac{dQ}{dR}(X_{n,(j)})}{\sum\limits_{j \in [n]} \frac{dQ}{dR}(X_{n,j})},
\]
where $X_{n,(1)} \leq \cdots \leq X_{n,(n)}$ are order statistics of $\{X_{n,j}\}_{j \leq n}$.
Since by \Cref{th: Th diff pval},
\[
\left| F_{n,n}\left( \frac{1}{2} \right) - F_{n+1,n+1}\left( \frac{1}{2} \right) \right| = O_P\left(\frac{1}{\sqrt{n}}\right),
\]
therefore, by the properties of the Cesaro limit, it suffices to show
\[
F_{2n,2n}\left( \frac{1}{2} \right) \xrightarrow[n \to \infty]{a.s.} \mu \quad \text{for some } \mu \neq \frac{1}{2}.
\]
Then
\[
\frac{1}{\lfloor cn + n^{1/4} \rfloor} \sum_{k=1}^{\lfloor cn + n^{1/4} \rfloor} F_{n,k}\left( \frac{1}{2} \right) \xrightarrow[n \to \infty]{a.s.} \mu \quad \text{with } \mu \neq \frac{1}{2},
\]
implying
\begin{align*}
    \liminf_{n \to \infty} \sup_{t \in [0,1]} \left| \frac{1}{\lfloor cn + n^{1/4} \rfloor} \sum_{k=1}^{\lfloor cn + n^{1/4} \rfloor} F_{n,k}(t) - t \right| &\geq \liminf_{n \to \infty} \left| \frac{1}{\lfloor cn + n^{1/4} \rfloor} \sum_{k=1}^{\lfloor cn + n^{1/4} \rfloor} F_{n,k}\left( \frac{1}{2} \right) - \frac{1}{2} \right| \\ &= |\mu - \frac{1}{2}| \\ &> 0.
\end{align*}
Thus \textbf{Step 1} follows if we show, 
\[
F_{2n,2n}\left( \frac{1}{2} \right) \xrightarrow[n \to \infty]{a.s.} \mu \quad \text{for some } \mu \neq \frac{1}{2}.
\]
To show that, observe for $n$ even:
\begin{align*}
    F_{n,n}\left( \frac{1}{2} \right) 
    &= \frac{ \sum\limits_{j = \frac{n}{2} + 1}^{n} \frac{dQ}{dR}(X_{n,j}) }
    { \sum\limits_{j \in [n]} \frac{dQ}{dR}(X_{n,j}) }.
\end{align*}
Now by Claim 1 of \Cref{Th:th3}, taking $\delta = \frac{1}{2}$, we have 
$\exists \, \tilde{\mu} > 0, \tilde{\mu} \neq 1$ such that:
\begin{align*}
\frac{\frac{1}{\delta n} \sum_{j \leq \delta n} \frac{dQ}{dR}(X_{n,(j)})}{\frac{1}{n} \sum_{j \in [n]} \frac{dQ}{dR}(X_{n,j}) }
    &\xrightarrow{\text{a.s.}} \tilde{\mu} \quad \text{as } n \to \infty.
\end{align*}
So,
\begin{align*}
    \frac{ \sum\limits_{j \leq \delta n} \frac{dQ}{dR}(X_{n,(j)}) }
    { \sum\limits_{j \in [n]} \frac{dQ}{dR}(X_{n,j}) }
    &\xrightarrow{\text{a.s.}} \frac{\tilde{\mu}}{2}.
\end{align*}
Thus with $\mu = 1 - \frac{\Tilde{\mu}}{2}$, $\mu \neq \frac{1}{2}$ we have 
\[  \frac{ \sum\limits_{j = \frac{n}{2} + 1}^{n} \frac{dQ}{dR}(X_{n,(j)}) }{ \sum\limits_{j \in [n]} \frac{dQ}{dR}(X_{n,j}) }
    \xrightarrow{\text{a.s.}} \mu, \]
showing that, 
\[ F_{2n,2n}\bigg(\frac{1}{2}\bigg) \xrightarrow{\text{a.s.}} \mu \quad \text{as} \quad n \to \infty. \]
This proves \textbf{Step 1}. 

\textbf{Step 2:} We will now show,

\[
\left\| \widehat{F}_{n, \lfloor cn + n^{1/4} \rfloor}(t) - \frac{1}{\lfloor cn + n^{1/4} \rfloor} \sum_{k=1}^{\lfloor cn + n^{1/4} \rfloor} \P_{\mathcal{G}_{n,k}}(p_{n,k} \leq t) \right\|_{\infty} = o(1) \quad \text{a.s.}
\]
\textbf{Proof of Step 2:} 
By Theorem 11.2 of \cite{vovk2005algorithmic}, 
\[
\P_{\mathcal{G}_{n,k}}(p_{n,k} \leq t) = t, \quad \forall t \in [0,1] \text{ for } k \leq cn,
\]
and $p_{n,1}, \ldots, p_{n,\lfloor cn \rfloor} \overset{iid}{\sim} U(0,1)$. So with $F$ the CDF of $U(0,1)$,

\begin{align*}
&\quad \left\| \widehat{F}_{n, \lfloor cn + n^{1/4} \rfloor}(t) - \frac{1}{\lfloor cn + n^{1/4} \rfloor} \sum_{k=1}^{\lfloor cn + n^{1/4} \rfloor} \P_{\mathcal{G}_{n,k}}(p_{n,k} \leq t) \right\|_{\infty} \\
&\leq \left\| \frac{1}{\lfloor cn + n^{1/4} \rfloor} \left( \sum_{k=1}^{\lfloor cn \rfloor} \mathbf{1}(p_{n,k} \leq t) - \lfloor cn \rfloor F(t) \right) \right\|_{\infty} \\
&\quad + \left\| \frac{1}{\lfloor cn + n^{1/4} \rfloor} \sum_{k=\lfloor cn \rfloor + 1}^{\lfloor cn + n^{1/4} \rfloor} \bigg( \mathbf{1}(p_{n,k} \leq t) - \P_{\mathcal{G}_{n,k}}(p_{n,k} \leq t) \bigg) \right\|_{\infty} \\
&\leq \frac{\lfloor cn \rfloor}{\lfloor cn + n^{1/4} \rfloor} \left\| \frac{1}{\lfloor cn \rfloor} \sum_{k=1}^{\lfloor cn \rfloor} \mathbf{1}(p_{n,k} \leq t) - F(t) \right\|_{\infty}
+ \frac{2n^{1/4}}{\lfloor cn + n^{1/4} \rfloor} \\
&= o(1) \quad \text{a.s.}
\end{align*}
Since by Glivenko–Cantelli lemma, as $p_{n,1}, \ldots, p_{n,\lfloor cn \rfloor} \overset{iid}{\sim} U(0,1)$, so
\[
\left\| \frac{1}{\lfloor cn \rfloor} \sum_{k=1}^{\lfloor cn \rfloor} \mathbf{1}(p_{n,k} \leq t) - F(t) \right\|_{\infty} = o(1) \quad \text{a.s.}
\]
Hence, \textbf{step 2} follows, and hence the theorem also follows.
\qed
\\
\medskip
\\
\textbf{Proof of \Cref{th: rel len CONCH}:} For $\alpha \in (0,1)$ let $C_{n,1 - \alpha}^\texttt{MCP}$ be the confidence set produced by the \texttt{MCP} algorithm applied on $\{X_{n,j}\}_{1 \leqslant j \leqslant n}$. By construction, 
   \[ C_{n,1 - \alpha}^\texttt{MCP} = \{j \in [n] : p^\texttt{CONF}_{n,j} > \alpha \}. \]
   Now following the notations as in Algorithm~\ref{alg: CONCH traiangular}, by Donsker's theorem (\cite{donsker1951invariance})
   \begin{equation}
   \label{eq: Donsker}
       F_n(t) \to \P(\sup_{x \in [0,1]}|\phi_x| \leqslant t), \quad t \geq 0,
   \end{equation}
   where $(\phi_x)_{x \geqslant 0}$ is the Brownian Bridge. 

For the first part, observe that by \Cref{th: Main theorem} for any $\Tilde{c} > 0$ , 
   \begin{equation}
       \begin{split}
           \liminf_{n \to \infty} \hspace{0.2cm} \text{KS}(\widehat{F}^{(0)}_{n, \lfloor cn + \Tilde{c}n^{1/4} \rfloor },u) > 0 \quad \text{a.s}, \\
           \liminf_{n \to \infty} \hspace{0.2cm} \text{KS}(\widehat{F}^{(1)}_{n, \lfloor cn - \Tilde{c}n^{1/4} \rfloor },u) > 0 \quad \text{a.s}.
       \end{split}
   \end{equation}
   Thus $W^{(0)}_{n, \lfloor cn + \Tilde{c}n^{1/4} \rfloor} \convAS \infty$ and $W^{(1)}_{n, \lfloor cn - \Tilde{c}n^{1/4} \rfloor} \convAS \infty$. Therefore by \Cref{eq: Donsker}, $p^\text{left}_{n,\lfloor cn + \Tilde{c}n^{1/4} \rfloor} \convAS 0$ and $p^\text{right}_{n,\lfloor cn - \Tilde{c}n^{1/4} \rfloor} \convAS 0$ as $n \to \infty$. Since by construction, $p^\texttt{CONF}_{n,t} = \min\{2P^\text{left}_{n,t},2p^\text{right}_{n,t},1\}$, therefore $p^\texttt{CONF}_{n,\lfloor cn + \Tilde{c}n^{1/4} \rfloor} \convAS 0$ and  $p^\texttt{CONF}_{n,\lfloor cn - \Tilde{c}n^{1/4} \rfloor} \convAS 0$ as $n \to \infty$. Let $A_n$ be the event 
   \begin{equation}
       A_{n,\Tilde{c}} := \big[p^\texttt{CONF}_{n,\lfloor cn - \Tilde{c}n^{1/4} \rfloor} \leqslant \alpha, \quad p^\texttt{CONF}_{n,\lfloor cn + \Tilde{c}n^{1/4} \rfloor} \leqslant \alpha\big].
   \end{equation}
 Then $\mathbf{1}(A_{n,\Tilde{c}}) \convAS 1$ as $n \to \infty$ and this holds for any $\Tilde{c} > 0$ and hence $\frac{l_{n,1-\alpha}}{n} \convAS 1$ as $n \to \infty$. Thus, part 1 of the theorem follows.

 For the second part, observe that $\widehat{\tau}_n \in C^\texttt{MCP}_{n,1 - \alpha}$ and thus $\mathbf{1}(\widehat{\tau}_n \in [cn - n^{1/4}, cn + n^{1/4}]) \convAS 1$ as $n \to \infty$. Thus $\frac{\widehat{\tau}_n}{\tau_n} \overset{\text{a.s.}}{=} 1 + o(1)$ and thus the theorem follows. \qed
\\
\medskip
\\
\textbf{Proof of \Cref{Th: optimising CONCH}:}
Observe that, using Theorem 5,
\[
\mathbf{1}\!\left\{\widehat{\tau}_n \in \big[\tau_n-n^{1/4},\,\tau_n+n^{1/4}\big]\right\}
\xrightarrow[n\to\infty]{a.s.} 1 .
\]

Let $f$ be a density and let $\widehat f_{\widehat{\tau}_n}$ denote the kernel density estimator of $f$
based on $\{X_{n,1},\ldots,X_{n,\widehat{\tau}_n}\}$, defined by
\[
\widehat f_{\widehat{\tau}_n}(x)
= \frac{1}{\widehat{\tau}_n h_{\widehat{\tau}_n}}\sum_{i=1}^{\widehat{\tau}_n}
K\!\left(\frac{x-X_{n,i}}{h_{\widehat{\tau}_n}}\right),
\qquad
\widehat f_{\tau_n}(x)
= \frac{1}{\tau_n h_{\tau_n}}\sum_{i=1}^{\tau_n}
K\!\left(\frac{x-X_{n,i}}{h_{\tau_n}}\right).
\]
Then
\[
\bigl\|\widehat f_{\widehat{\tau}_n}-f\bigr\|_\infty
\le
\bigl\|\widehat f_{\widehat{\tau}_n}-\widehat f_{\tau_n}\bigr\|_\infty
+\bigl\|\widehat f_{\tau_n}-f\bigr\|_\infty .
\]

By \Cref{assn: kernel} and Theorem~2 of \cite{wied2012consistency},
\[
\bigl\|\widehat f_{\tau_n}-f\bigr\|_\infty=o(1)\quad\text{a.s.}.
\]

For the other term, write
\begin{align*}
\bigl\|\widehat f_{\widehat{\tau}_n}-\widehat f_{\tau_n}\bigr\|_\infty
&\le
\left\|
\frac{1}{\widehat{\tau}_n h_{\widehat{\tau}_n}}
\sum_{i=1}^{\widehat{\tau}_n} K\!\left(\frac{x-X_{n,i}}{h_{\widehat{\tau}_n}}\right)
-
\frac{1}{\tau_n h_{\widehat{\tau}_n}}
\sum_{i=1}^{\tau_n} K\!\left(\frac{x-X_{n,i}}{h_{\widehat{\tau}_n}}\right)
\right\|_\infty\\
&\quad+
\left\|
\frac{1}{\tau_n h_{\widehat{\tau}_n}}
\sum_{i=1}^{\tau_n} K\!\left(\frac{x-X_{n,i}}{h_{\widehat{\tau}_n}}\right)
-
\frac{1}{\tau_n h_{\tau_n}}
\sum_{i=1}^{\tau_n} K\!\left(\frac{x-X_{n,i}}{h_{\tau_n}}\right)
\right\|_\infty\\
&\le
\frac{1}{\widehat{\tau}_n h_{\widehat{\tau}_n}}
\sum_{i=\widehat{\tau}_n\wedge\tau_n}^{\widehat{\tau}_n\vee\tau_n}
\Bigl|K\!\left(\frac{x-X_{n,i}}{h_{\widehat{\tau}_n}}\right)\Bigr|
\;+\;
\tau_n\left|
\frac{1}{\widehat{\tau}_n h_{\widehat{\tau}_n}}
-\frac{1}{\tau_n h_{\tau_n}}
\right|.
\end{align*}
Let $k=\sup_{x\in\mathbb{R}}|K(x)|<\infty$ (since $K:\mathbb{R}\to\mathbb{R}$ is continuous and
$\lim_{|x|\to\infty}K(x)=0$). Then
\[
\bigl\|\widehat f_{\widehat{\tau}_n}-\widehat f_{\tau_n}\bigr\|_\infty
\le
k\,\frac{|\widehat{\tau}_n-\tau_n|}{\widehat{\tau}_n h_{\widehat{\tau}_n}}
+
k\,\tau_n\left|
\frac{1}{\widehat{\tau}_n h_{\widehat{\tau}_n}}
-\frac{1}{\tau_n h_{\tau_n}}
\right|
=
\Theta\!\left(n^{-(3/4+\alpha)}\right)
=o(1)\quad\text{a.s.,}
\]
where $\alpha\in(-1/2,0)$.
Hence,
\[
\bigl\|\widehat f_{\widehat{\tau}_n}-f\bigr\|_\infty=o(1).
\]

Applying this to $q(x)$ computed on
$\{X_{n,1},\ldots,X_{n,\widehat{\tau}_n}\}$ and to $r(x)$ computed on
$\{X_{n,\widehat{\tau}_n+1},\ldots,X_{n,n}\}$, we obtain for fixed $x,y\in\mathbb{R}$,
\[
\mathbf{1}\!\left\{\widehat s_n^*(x) > \widehat s_n^*(y)\right\}
\convAS\mathbf{1}\!\left\{s^*(x)<s^*(y)\right\},
\qquad
\mathbf{1}\!\left\{\widehat s_n^*(x)=\widehat s_n^*(y)\right\}
\convAS\mathbf{1}\!\left\{s^*(x)=s^*(y)\right\}.
\]

Let $X_1\sim R$ and $Y_1\sim Q$ be independent. Then
\begin{align*}
\E\!\left[P_n[s]\right]
&=
\P\!\left(s(X_1)>s(Y_1)\right)
+\frac12\,\P\!\left(s(X_1)=s(Y_1)\right).
\end{align*}
Hence, by the dominated convergence theorem,
\begin{align*}
\P\!\left(\widehat s_n^*(X_1)>\widehat s_n^*(Y_1)\right)
&\to
\P\!\left(s^*(X_1)>s^*(Y_1)\right),\\
\P\!\left(\widehat s_n^*(X_1)=\widehat s_n^*(Y_1)\right)
&\to
\P\!\left(s^*(X_1)=s^*(Y_1)\right).
\end{align*}
Thus,
\[
\E\!\left[P_n[\widehat s_n^*]\right]-\E\!\left[P_n[s^*]\right]=o(1),
\qquad n\to\infty .
\]
and the theorem follows.\qed

\end{document}